\documentclass[11pt]{amsart}
\usepackage{color}
\usepackage{enumerate}
\usepackage{hyperref} 

\makeatletter 
\def\@tocline#1#2#3#4#5#6#7{\relax
  \ifnum #1>\c@tocdepth 
  \else
    \par \addpenalty\@secpenalty\addvspace{#2}%
    \begingroup \hyphenpenalty\@M
    \@ifempty{#4}{%
      \@tempdima\csname r@tocindent\number#1\endcsname\relax
    }{%
      \@tempdima#4\relax
    }%
    \parindent\z@ \leftskip#3\relax \advance\leftskip\@tempdima\relax
    \rightskip\@pnumwidth plus4em \parfillskip-\@pnumwidth
    #5\leavevmode\hskip-\@tempdima
      \ifcase #1
       \or\or \hskip 1em \or \hskip 2em \else \hskip 3em \fi%
      #6\nobreak\relax
    \dotfill\hbox to\@pnumwidth{\@tocpagenum{#7}}\par
    \nobreak
    \endgroup
  \fi}
\makeatother

\newtheorem{theorem}{Theorem}[section]
\newtheorem{lemma}[theorem]{Lemma}
\newtheorem{prop}[theorem]{Proposition}
\newtheorem{cor}[theorem]{Corollary}

\newtheorem{corollary}[theorem]{Corollary}

\theoremstyle{definition}
\newtheorem{definition}[theorem]{Definition}

\newtheorem{remark}[theorem]{Remark}

\newcommand{\R}{\mathbb R}

\newcommand{\C}{\mathbb C}

\newcommand{\N}{\mathbb N}

\def\Im{{\sf Im}\,}

\newcommand{\scalar}[1]{\langle #1\rangle}

\newcommand{\diff}[1]{\partial_{#1}}

\newcommand{\norm}[1]{| #1|}

\newcommand{\mR}{\mathbb{R}}
\newcommand{\mC}{\mathbb{C}}
\newcommand{\mN}{\mathbb{N}}

\newcommand{\mZ}{\mathbb{Z}}

\newcommand{\cH}{\mathcal{H}}

\theoremstyle{remark} 

\numberwithin{equation}{section}

\begin{document}

\title[]{Implementing zonal harmonics with the Fueter principle}

\author[A. Altavilla]{A. Altavilla${}^{\ddagger}$}\address{Altavilla Amedeo: Dipartimento di Ingegneria Industriale e Scienze Matematiche, Universit\`a Politecnica delle Marche, Via Brecce Bianche, 60131,
Ancona, Italy} \email{altavilla@dipmat.univpm.it}

\author[Hendrik De Bie]{H. De Bie${}^{\dagger}$}\address{Hendrik De Bie: Department of Mathematical Analysis, Faculty of Engineering and Architecture,
Ghent University, Krijgslaan 281-S8, 9000 Gent, Belgium} \email{Hendrik.DeBie@UGent.be}

\author[Michael Wutzig]{M. Wutzig${}^{@}$}\address{Michael Wutzig: Institute of Applied Analysis, Faculty of Mathematics and Computer Science,
Freiberg University of Mining and Technology, Pr\"{u}ferstra{\ss}e 4, 09599 Freiberg, Germany} \email{michael.wutzig@gmail.com}

\thanks{${}^{\ddagger}$GNSAGA of INdAM, SIR grant {\sl ``NEWHOLITE - New methods in holomorphic iteration''} n. RBSI14CFME and SIR grant {\sl AnHyC - Analytic aspects in complex and hypercomplex geometry} n. RBSI14DYEB. The first author was also financially supported by a INdAM fellowship ``{\em mensilit\`a
di borse di studio per l'estero a.a. 2018--2019}'' and wants to thanks the Clifford Research Group at Ghent University where this fellowship has been spent.
${}^{\dagger}$HDB was supported by the Research Foundation 
Flanders (FWO) under Grant EOS 30889451. 
}

\date{\today }


\begin{abstract} 
By exploiting the Fueter theorem, we give new formulas to compute zonal harmonic functions 
in any dimension. We first give a representation of them as a result of a suitable ladder operator
acting on a constant function. Then, inspired by recent work of A.~Perotti, using techniques
from slice regularity, we derive explicit expressions for zonal harmonics starting from the 2 and 3 dimensional cases.
It turns out that all zonal harmonics in any dimension are related to the real part of powers of the standard Hermitian product in $\mathbb{C}$. At the end we compare formulas, obtaining interesting equalities involving the real part of positive and negative powers of the standard Hermitian product.

In the two appendices we show how our computations are optimal compared to direct ones.
\end{abstract}
\maketitle
\tableofcontents

\section{Introduction}

The aim of this paper is to give new ways to implement zonal harmonic functions in $\mR^{N}$, for $N\geq 2$, adapting
the celebrated \textit{Fueter's theorem}~\cite{fueter,qian} to our setting. In the field of \textit{Clifford analysis}, researchers
refer to Fueter's theorem anytime it is possible to construct regular hypercomplex functions starting from suitable modifications of complex holomorphic ones, possibly hit by some (fractional) power of the usual Laplacian. 
This general principle fits well in our context, as the original proof by Fueter passes in fact through harmonic analysis (see also~\cite{qiansommen}).
In our work, inspired by recent results by A.~Perotti~\cite{perotticlifford}, we give a complete description 
of all zonal harmonic functions in any dimension, knowing only their characterization in dimension 2 and 3.
Starting from these well known cases, we first rephrase them using complex functions to which we will apply 
the Fueter principle obtaining, at first, Clifford algebra-valued functions of a \textit{paravector} variable (see the following subsection) and eventually our goal. 

We now define and introduce the material that will be used and discussed later.

\subsection{Clifford algebras: definitions and related function theories}

	Starting from the real vector space $\mR^n$ with orthonormal basis $\{e_1,\cdots,e_n\}$  we construct the real Clifford algebra $\R_n$. Algebraic multiplication follows the rules
		\begin{align*}
			e_je_k+e_ke_j=-2\delta_{jk}\hspace{15pt}j,k=1,\cdots,n.
		\end{align*}
		Elements $x\in\R_n$ can be written in the form $x=\sum\limits_Ae_Ax_A$, 
with $x_A\in\mR$. Summation runs over all possible ordered index sets $A=\{a_j\}_{j=1}^k$, $1\leq a_1<\cdots<a_k\leq n$, and $e_A$ denotes the 
		$k$-vector $e_A=e_{a_1}\cdots e_{a_k}$.
In particular the real part of $x\in\mR_{n}$ is defined as $x_{0}:=x_{\emptyset}$.
In $\R_{n}$ we consider the subspace of \textit{paravectors} $\mR^{n+1}$ defined as 
$$\R^{n+1}:=\left\{x=x_{0}+\underline{x}\,\big|\,\underline{x}=\sum_{k=1}^{n}x_{k}e_{k},\,x_{k}\in\mathbb{R},\, k=0,\dots,n\right\}.$$

For an element $x\in\R_{n}$, the (\textit{Clifford}) conjugation is denoted as 
$x^{c}$, where for any basis element $e_{i}^{c}=-e_{i}$. Notice that, if $x=x_{0}+\underline{x}\in\R^{n+1}$
is a paravector, then $x^{c}=x_{0}-\underline{x}$.
In $\R_{n}$ it is possible to define the \textit{sphere of imaginary units} $\mathbb{S}$ and the \textit{quadratic cone} $\mathcal{Q}$ as
$$
\mathbb{S}:=\{x\in \R_{n}\,|\,x+x^{c}=0,\, xx^{c}=1\},\qquad\mathcal{Q}:=\bigcup_{J\in \mathbb{S}}\C_{J},
$$
where $\C_{J}$ denotes the complex plane generated by $1$ and $J$. Notice that $\mathbb{R}^{n+1}\subseteq\mathcal{Q}$, for any $n>1$ and equality holds only for $n=1$. Moreover, while for $n=1$ and $n=2$ the quadratic cone coincides with the Clifford algebra $\R_{n}$, for $n>2$, the set of paravectors is a proper subset of the quadratic cone
(see e.g.~\cite{ghiloniperotti}). Summarizing, for any $n>2$ we have
$$
\mR^{n+1}\subset \mathcal{Q}\subset\mR_{n}.
$$
The quadratic cone is the natural subset of $\R_{n}$ where slice regular functions are defined (see Subsection~\ref{slice}).

Given two paravectors $x=x_{0}+\underline{x}$ and $y=y_{0}+\underline{y}\in\R^{n+1}$, their product can be written as
$$
xy=x_{0}y_{0}-\langle \underline{x},\underline{y}\rangle + x_{0}\underline{y}+y_{0}\underline{x}+\underline{x}\wedge \underline{y},
$$
where $\langle,\rangle$ stands for the usual Euclidean product in $\R^{n}$ and, in general $\underline{x}\wedge \underline{y}\notin \mathbb{R}^{n+1}$. In fact $\underline{x}\wedge \underline{y}\in\mathbb{R}^{n+1}$ if and only if
$\underline{x}\wedge \underline{y}=0$ if and only if $\underline{x}$ and  $\underline{y}$ are linearly independent over $\mathbb{R}$.
Hence, we have that $xx^{c}=x_{0}^{2}+\langle \underline{x},\underline{x}\rangle=\langle \underline{x},\underline{x}\rangle=|x|^{2}$, i.e. the standard square norm in $\R^{n+1}$. 
Moreover we also have that $x^{2}=x_{0}^{2}-|\underline{x}|^{2}+2x_{0}\underline{x}$ and that $\underline{x}^{2}=-|\underline{x}|^{2}=-|\underline{x}^{c}|^{2}=(\underline{x}^{c})^{2}$.
Given any nonzero paravector $x\in\mathbb{R}^{n+1}$, its inverse is given by $x^{-1}=\frac{x^{c}}{|x|^{2}}$.
Therefore, if we denote by $\langle,\rangle_{N}$ the standard scalar product in $\mathbb{R}^{N}$, given $x,y\in\R^{n+1}$, with $y\neq 0$, then 
$$(xy^{-1})_{0}=\frac{x_{0}y_{0}+\langle \underline{x},\underline{y}\rangle_{n}}{|y|^{2}}=\frac{\langle x,y\rangle_{n+1}}{|y|^{2}},$$
which is exactly the \textit{Fourier coefficient} of $x$ with respect to $y$.
In particular, for any couple of paravectors $x$ and $y$, we have that $(xy^{c})_{0}=x_{0}y_{0}+\langle \underline{x},\underline{y}\rangle=\langle x,y\rangle_{n+1}$. In the remainder of the paper we will drop the subscript $n+1$ whenever it is unambiguous.

In the context of Clifford algebras we will deal with two theories of regular functions, namely \textit{monogenicity} and
\textit{slice regularity}. While the first one takes its origin in the work of R.~Fueter~\cite{fueter} and Gr.~Moisil and N.~Th\'eodoresco~\cite{MT} in the 1930's, slice regularity was introduced some ten years ago by G.~Gentili and D.~Struppa~\cite{gentilistruppa}. We now quickly introduce both theories and highlight the main features we are interested in. For more information see~\cite{BDS,CSSbook, DSS,gentilistoppatostruppa}.

\subsubsection{Monogenicity}

We now introduce the Dirac operator $\partial_{\underline{x}}$ and the generalized Cauchy-Riemann operators $D$ and $\bar D$ as
$$
\partial_{\underline{x}}=e_{1}\frac{\partial}{\partial x_{1}}+\dots+e_{n}\frac{\partial}{\partial x_{n}},\qquad D=\frac{\partial}{\partial x_{0}}-\partial_{\underline{x}},\qquad \bar D=\frac{\partial}{\partial x_{0}}+\partial_{\underline{x}}.
$$
We have
$$ D \bar D = - \Delta_{x,n+1}$$
with $\Delta_{x,n+1}$ the Laplacian in $\mR^{n+1}$ with respect to the variable $x$. We will omit the subindex where it is permitted by the context.
 A function that satisfies $D f=0$ is called \textit{monogenic}. Moreover, a homogeneous polynomial in the variable $x=x_{0}+\underline{x}$ (i.e. a homogeneous polynomial in the variables $(x_{0},x_{1},\dots,n_{n})$), that is monogenic is called a \textit{spherical monogenic}.

\subsubsection{Slice regularity}\label{slice}

The concept of slice regularity on Clifford algebras was introduced in~\cite{colombosabadinistruppa} for functions defined over $\mathbb{R}^{n+1}$.
However, we will use the approach of \textit{stem functions} introduced by R.~Ghiloni and A.~Perotti~\cite{ghiloniperotti} and inspired by the Fueter theorem~\cite{fueter}.


Let $D$ be any subset of $\C$ such that $\bar D=D$. A \textit{stem function} is a function
$F:D\to \R_{n}\otimes\C$, such that, for any $\alpha+i\beta\in D$ it holds that $F(\alpha+i\beta)=\overline{F(\alpha+i\beta)}$, i.e. if $F=F_{1}+\imath F_{2}$, with $F_{1}$ and $F_{2}$ being $\R_{n}$-valued functions, then $F_{1}(\alpha-i\beta)+\imath F_{2}(\alpha-i\beta)=F_{1}(\alpha+i\beta)-\imath F_{2}(\alpha+i\beta)$. The last equality is equivalent to $F_{1}$ being even and $F_{2}$ being odd both with respect to $\beta$.

Let now $D\subset\C$ be any open domain such that $D=\bar D$. We call the \textit{circularization of} $D$
the set 
$$
\Omega_{D}:=\{\alpha+I\beta\in\mathcal{Q}\,|\, \alpha+i\beta\in D,\, I\in\mathbb{S}\}.
$$
Any subset of the quadratic cone $\Omega\subset\mathcal{Q}$ such that $\Omega\setminus \R$ is symmetric with respect to $\R$
will be called a \textit{circular set}.  Clearly, for any circular set $\Omega$ we have $\Omega=\Omega_{\Omega\cap\C_{J}}$, for any $J\in\mathbb{S}$.

\begin{definition}
Let $\Omega_{D}$ be a circular domain and $F:D\to\R_{n}\otimes \C$ be a holomorphic stem function.
Then $f=\mathcal{I}(F):\Omega_{D}\to\R_{n}$ defined as 
$$
f(\alpha+I\beta)=F_{1}(\alpha+i\beta)+IF_{2}(\alpha+i\beta)
$$
is said to be a \textit{slice function} (\textit{induced by} $F$). If $F$ is holomorphic with respect to the standard
complex structures of $\C$ and $\R_{n}\otimes \C$, then $f$ is said to be \textit{slice regular}.
If $F_{1}$ and $F_{2}$ are $\R$-valued, then $f$ is said to be \textit{slice preserving}.
The set of slice regular functions defined on $\Omega$ will be denoted by $\mathcal{S}(\Omega)$, while
the set of slice preserving regular functions by $\mathcal{S}_{\R}(\Omega)$.
\end{definition}
 In~\cite{ghiloniperotti} it is shown that any slice function is induced by a unique stem function.
Notice that any complex intrinsic holomorphic function $\Phi:D\to\C$ can be seen
as a stem function inducing a slice preserving function $\phi:\Omega_{D}\to\R_{n}$, and vice versa.

Examples of slice regular functions are polynomials and converging power series of the form $f(x)=\sum x^{k}a_{k}$, with $a_{k}\in\R_{n}$; examples of slice preserving functions are polynomials and power series
with real coefficients.
Notice that a slice function $f$ is slice preserving if and only if $f(\Omega\cap\C_{J})\subset\C_{J}$, for any $J\in\mathbb{S}$. Slice regular functions which preserve one or all slices of the form $\C_{J}$ are studied in detail in~\cite{AdF1,AdF2}.

\subsubsection{Fueter-Sce-Qian Theorem}
We now recall the statement of the Fueter theorem in the form of T.~Qian~\cite{qian}. To do this, we need to 
recall the definition and the main properties of the fractional Laplacian.

The Fourier transform with respect to the variable $x$ and its inverse is defined as
$$
\mathcal{F}(\varphi)(\xi):=\frac{1}{(2\pi)^{(n+1)/2} }\int_{\R^{n+1}}\varphi(x)e^{i\langle x,\xi\rangle} dx,\quad \mathcal{F}^{-1}(\psi)(x):=\frac{1}{(2\pi)^{(n+1)/2} }\int_{\R^{n+1}}\psi(x)e^{-i\langle x,\xi\rangle} dx.
$$
With these definitions it is possible~\cite{stein} to compute the $\mu$-Laplacian as
$$
\Delta_{x,n+1}^{\mu}\varphi:=\mathcal{F}^{-1}\left((-i|\xi|)^{2\mu}\mathcal{F}(\varphi)(\xi)\right)
$$

Let now $\Delta^{\mu}$ denote $\Delta_{x,n+1}^{\mu}$, for any $\mu>0$.
Thanks to the definition, for any differentiable function $f$ such that $\Delta^{\mu}f$ is a differentiable function,
the partial derivative and the multiplication by a scalar commute with $\Delta^{\mu}$. Under these hypotheses one therefore has
\begin{equation}\label{commuting}
[\Delta^{\mu},\partial_{\underline{x}}]f=[\Delta^{\mu},\bar D]f=0.
\end{equation}

Moreover, thanks to~\cite[formula (6) p. 118]{stein}, for any $\alpha, \beta >0$, such that $\alpha+\beta<n$, whenever it makes sense, we have
$$
\Delta^{\alpha}(\Delta^{\beta}f)=\Delta^{\alpha+\beta}f.
$$

We know that if $p_{k}$ is a homogeneous function of degree $k$, then, whenever it is well defined, the function $\Delta^{\mu}p_{k}$ is homogeneous of degree $k-2\mu$ (a generalization of this formula was given~\cite[Appendix B]{brascofranzina} where our particular case is given for $p=2$ and $s=1/2$).

With our choice of coefficients in the definition of the Fourier transform, according to~\cite[formula 32, p. 73]{stein} we have
\begin{equation}\label{fourierpol}
\mathcal{F}\left(\frac{P_{k}(.)}{|.|^{k+n+1-a}}\right)(\xi)=\gamma_{k,a}\frac{P_{k}(\xi)}{|\xi|^{k+a}},
\end{equation}
where $0<a<n+1$, $k\in\N$, $P_{m}$ is a homogeneous harmonic polynomial of degree $k$ and
\begin{equation}\label{gamma}
\gamma_{k,a}:=i^{k}\Gamma\left(\frac{k+a}{2}\right)\bigg/\Gamma\left(\frac{k+n+1-a}{2}\right).
\end{equation}

For any function defined in the space of paravectors it is possible to define the following two inversion operators
with respect to the variable $x$:
\begin{itemize}
\item the Clifford inversion 
$$
\mathfrak{K}[f](x):=\frac{(x_{0}-\underline{x})}{|x|^{n+1}}f\left(\frac{x_{0}-\underline{x}}{|x|^{2}}\right);
$$
\item the Kelvin inversion 
\begin{equation}
\label{kel}
\mathcal{K}[f](x):=\frac{1}{|x|^{n-1}}f\left(\frac{x}{|x|^{2}}\right).
\end{equation}
\end{itemize}
By direct computation we have that $\mathfrak{K}^{2}=\mathcal{K}^{2}=1$.
If $p_{k}$ is a homogeneous function of degree $k$, then
\begin{equation}\label{kelvinhomo}
\mathcal{K}[p_{k}(x)]=|x|^{2-(n+1)-2k}p_{k}(x)
\end{equation}
It is known that Clifford inversion preserves monogenicity, while Kelvin inversion preserves harmonicity.
In the following theorem we recall the statement of the \textit{Fueter theorem} in Qian's formulation~\cite{qian}, (see also \cite{dkqs,dongqian,kou} for some other reformulations).

\begin{theorem}\label{fueterqian}
Let $k$ be a positive integer and $\Delta=\Delta_{x,n+1}$ denote the usual Euclidean Laplacian with respect to the variable $x$. If $x\in\R^{n+1}$ is a paravector, then
\begin{enumerate}
\item $\Delta^{\frac{n-1}{2}}(x^{-k})$ is monogenic and homogeneous of degree $1-n-k$;
\item $\mathfrak{K}[\Delta^{\frac{n-1}{2}}(x^{-k})]$ is monogenic and homogeneous of degree $k-1$;
\item if $n=2h+1$ is odd then 
$$\mathfrak{K}[\Delta^{h}(x^{-k})]=\Delta^{h}x^{n+k-2}.$$
\end{enumerate}
\end{theorem}

\subsection{Zonal harmonics and  Gegenbauer polynomials}
We consider complex-valued functions defined on $\mR^{N}$.
In harmonic analysis a \textit{spherical harmonic} is a homogeneous polynomial of some degree $k$ that is in the 
kernel of the Laplacian. In the space $\mathcal{H}_{k}$ of spherical harmonics of degree $k$ it is possible to 
define the following $L^{2}$ inner product:
$$
\langle P(x),Q(x)\rangle_{\mathbb{S}}:=\frac{1}{\omega_{N-1}}\int_{\mathbb{S}^{N-1}}\overline{P(x)}Q(x)d\sigma(x),
$$
where $d\sigma(x)$ is the usual  Lebesgue measure on the sphere $\mathbb{S}^{N-1}$ and $\omega_{N-1}$ denotes
its surface area. Fix a point $y\in\mathbb{S}^{N-1}$, and consider the linear map $\Lambda:\mathcal{H}_{k}\to\mathbb{C}$ defined as $\Lambda(P)=P(y)$. Since $\mathcal{H}_{k}$ is a finite dimensional inner product space,
for any $N\in\mathbb{N}$, there exists a \textit{reproducing kernel}, i.e. a unique function $Z_{k}^{\lambda(N)}(\cdot,y)\in\mathcal{H}_{k}$ such that
$$
P(y)=\Lambda(P)=\langle P, Z_{k}^{\lambda(N)}(\cdot,y)\rangle_{\mathbb{S}},\qquad \forall P\in\mathcal{H}_{k}.
$$
Such a function $Z_{k}$ is called the \textit{zonal harmonic} of degree $k$ with pole $y$ (see~\cite{axler,SW}). For reproducing kernels in the context of Clifford analysis see~\cite{DBSW1}.

In view of our interest in giving new representations of zonal harmonic functions, we now discuss
their explicit expression in terms of Gegenbauer polynomials~\cite{stein,SW}.

%
		The Gegenbauer polynomial of degree $k$ and parameter $\lambda\in\mR$, $\lambda>-\frac{1}{2}$ is given explicitly by
		\begin{align}\label{explicitgegenbauer}
			C_k^\lambda(z)=\sum_{j=0}^{\lfloor\frac{k}{2}\rfloor}(-1)^j\frac{\Gamma(k-j+\lambda)}{\Gamma(\lambda)j!(k-2j)!}(2z)^{k-2j}.
		\end{align}
The derivative of a Gegenbauer polynomial is (see~\cite[(4.7.14)]{Sz})
\begin{equation}\label{G1}
\frac{d}{dz}C^{\lambda}_{k}(z)=2\lambda C^{\lambda+1}_{k-1}(z).
\end{equation}
The recursion relation is (see~\cite[Section 10.9, formula (13)]{Bateman})
\begin{equation}\label{G3}
zC^{\lambda+1}_{k-1}(z)=\frac{k}{2(k+\lambda)}C^{\lambda+1}_{k}(z)+\frac{k+2\lambda}{2(k+\lambda)}C^{\lambda+1}_{k-2}(z).
\end{equation}
We also have (see~\cite[(4.7.27)]{Sz})
\begin{align}\label{g2}
4\lambda(\ell+\lambda+1)(1-t^2)C_\ell^{\lambda+1}(t)&=(\ell+2\lambda)(\ell+2\lambda+1)C_\ell^\lambda(t)-(\ell+1)(\ell+2)C_{\ell+2}^\lambda(t),
\end{align}
and~\cite[(4.7.29)]{Sz} 
\begin{equation}\label{G3bis}
\frac{\lambda+ m}{\lambda}C^{\lambda}_{m}(z)=C^{\lambda+1}_{m}(z)-C^{\lambda+1}_{m-2}(z).
\end{equation}
Finally, from~\cite[Section 10.9, formula (24)]{Bateman} we have
\begin{equation}\label{G4}
z\frac{d}{dz}C^{\lambda}_{k}(z)-kC^{\lambda}_{k}(z)=\frac{d}{dz}C^{\lambda}_{k-1}\quad\Leftrightarrow\quad
(2\lambda)zC^{\lambda+1}_{k-1}(z)-kC^{\lambda}_{k}(z)=2\lambda C^{\lambda+1}_{k-2}(z).
\end{equation}

For $x,y\in\R^{n+1}$, we set the following notation,
\begin{equation}\label{angle}
w:=\frac{\langle x,y\rangle}{|x||y|}.
\end{equation}
When $\lambda=\frac{n+1}{2}-1=\frac{n-1}{2}$ and choosing the suitable variable, Gegenbauer polynomials represent zonal harmonics $Z_{k}^{\lambda}$ in $\mR^{n+1}$ in the following way~\cite{stein,SW}
\begin{equation}\label{zonalGegenbauer}
Z_{k}^{\lambda}(x,y)=\frac{k+\lambda}{\lambda}C_{k}^{\lambda}(w)(|x||y|)^{k}.
\end{equation}
If the total dimension $n+1$ equals 3, then we deal with the classical case of the study of spherical harmonics. 
In this case, Gegenbauer polynomials $C_{k}^{\frac{1}{2}}$ coincide with the Legendre polynomials.

Notice that when $n+1=2$ (and so for $\lambda=0$), the previous formula is not well defined. However~\cite[(4.7.8)]{Sz}, it holds:
\begin{equation}\label{chebyshev}
Z_{k}^{0}(x,y)=\lim_{\lambda\to 0}\frac{k+\lambda}{\lambda}C_{k}^{\lambda}(w)(|x||y|)^{k}=2T_{k}(w)(|x||y|)^{k},
\end{equation}
where $T_{k}$ is the Chebyshev polynomial of the first kind of degree $k$. Chebyshev polynomials satisfy the following useful equality, which is a consequence of~\eqref{G3bis}:
\begin{equation}\label{chebygegen}
2T_{k}(t)=C^{1}_{k}(t)-C^{1}_{k-2}(t).
\end{equation}
This particular case is described in more detail in the following subsection.
\subsection{Dimension 2}
When the dimension is equal to $2$, it is easy to compute zonal harmonics. For any $k\geq 1$, if $x=\rho_{1}e^{i\vartheta_{1}},y=\rho_{2}e^{i\vartheta_{2}}\in\mR^{2}=\mR_{2}=\mC$, these are given by~\cite[p. 94]{axler}
$$
Z^{0}_{0}\equiv 1,\qquad Z^{0}_{k}(x,y)=\left(\rho_{1}\rho_{2}\right)^{k}2\cos\left(k(\vartheta_{1}-\vartheta_{2})\right)=2T_{k}(w)(|x||y|)^{k}.
$$
Notice that $Z^{0}_{k}$ can also be obtained in the following two ways (recall that, for any $k\in\mZ\setminus\{0\}$ the function $\left(\left(xy^{c}\right)^{k}\right)_{0}$ is a homogeneous polynomial of degree $k$ in $x$):
\begin{equation}\label{2dim1}
\left((xy^{c})^{k}\right)_{0} =\left(\rho_{1}\rho_{2}\right)^{k}\left(\exp\left(ik(\vartheta_{1}-\vartheta_{2})\right)\right)_{0}=\left(\rho_{1}\rho_{2}\right)^{k}\cos\left(k(\vartheta_{1}-\vartheta_{2})\right)=\frac{1}{2}Z^{0}_{k}(x,y),
\end{equation}
and,
\begin{equation}\label{2dim2}
\begin{split}
\mathcal{K}\left[\left(\left(xy^{-1}\right)^{-k}\right)_{0}\right](x)& =\mathcal{K}\left[\rho_{2}^{2k}\left(\left(xy^{c}\right)^{-k}\right)_{0}\right](x)= \rho_{2}^{2k}\rho_{1}^{2k}\left(\left(xy^{c}\right)^{-k}\right)_{0}\\
&= \rho_{2}^{2k}\rho_{1}^{2k}\left(\left(\frac{yx^{c}}{\rho_{2}^{2}\rho_{1}^{2}}\right)^{k}\right)_{0}=\left(\left(yx^{c}\right)^{k}\right)_{0}=\left(\left(\left(xy^{c}\right)^{c}\right)^{k}\right)_{0}\\
&=\left(\left(\left(xy^{c}\right)^{k}\right)^{c}\right)_{0}=\left((xy^{c})^{k}\right)_{0}=\frac{1}{2}Z^{0}_{k}(x,y).
\end{split}
\end{equation}
Notice, in particular, that, for any integer $k\neq 0$
$$
\left((xy^{c})^{k}\right)_{0}=\mathcal{K}\left[\left(\left(xy^{-1}\right)^{-k}\right)_{0}\right](x).
$$

\subsection{Main results and outline of the paper}
Let $\scalar{y,\nabla_x}$ denote the partial derivative in the $y$ direction.
In Section~\ref{ladder} we study the action of the \textit{ladder operator} defined as $\mathcal{K}\scalar{y,\nabla_x}\mathcal{K}$ on the set of Gegenbauer polynomials. If $x,y$ belong to $\mR^{n+1}$, it turns out that, up to an explicit coefficient, the ladder
operator applied to $C^{\lambda}_{\ell}(w)(|x||y|)^{\ell}$, returns $C^{\lambda}_{\ell+1}(w)(|x||y|)^{\ell+1}$. Hence,
starting from the zeroth Gegenbauer polynomial of parameter $\lambda$ (and hence from the zeroth zonal harmonic in any dimension), we can derive the Gegenbauer polynomial of the same parameter and of any degree $k$. In particular,
the following equality is proven in Theorem~\ref{th1}
\begin{align*}
	\left(\mathcal{K}\scalar{y,\nabla_x}\mathcal{K}\right)^k\left[1\right]=(-1)^{k}k!\frac{\lambda}{\lambda+k}Z^{\lambda}_{k}(x,y).
\end{align*}
This representation will be used twice: to derive a new formula for the Poisson kernel by means of
the formal exponential of $\mathcal{K}\scalar{y,\nabla_x}\mathcal{K}$ and later in the last result of Section~\ref{further}. 

In Section~\ref{iterated} we give a first demonstration of the Fueter principle (in the M.~Sce flavor~\cite{sce}).
Starting from zonal harmonics in dimension 2 and 3, we are able, by applying a proper amount of Laplacians,
to obtain zonal harmonics in any even and odd dimension, respectively. In particular, we prove the following two equalities
\begin{itemize}
\item if $x,y\in\mR^{2m+3}$, then,
\[
(\Delta_{y}\Delta_{x})^{m}\left[ Z^{1/2}_{k+2m}(x,y)\right]=\widetilde{\beta}_{k}^{m} Z^{1/2+m}_{k}(x,y), 
\]
\item if $x,y\in\mR^{2m+2}$, then,
\[
(\Delta_{y}\Delta_{x})^{m}\left[ Z^{0}_{k+2m}(x,y)\right]=\widehat{\beta}_{k}^{m} Z^{m}_{k}(x,y).
\]
\end{itemize}
where $\widetilde{\beta}_{k}^{m}$ and $\widehat{\beta}_{k}^{m}$ are constants, depending only on the dimension and the degree, and are explicitly given in Corollaries~\ref{odddim} and~\ref{evendim}.
The proof of the previous equalities is not given by direct computation but by making use of properties of Gegenbauer polynomials. In Appendix~\ref{iterate} we show how direct computations might become long and difficult very quickly.

After some preliminaries about slice regular functions defined on Clifford algebras and their harmonicity properties, in Section~\ref{slicereg}
we give an explicit and complete proof of the Fueter theorem for this particular class. In particular, in Lemma~\ref{fueterslice} we prove that, for any positive integer $n$, given any slice regular function $f:\Omega\subset\mR^{n+1}\to\R_{n}$, whenever it is well defined, the function $\Delta^{\frac{n-1}{2}}f$ is monogenic.
Moreover, whenever it is well defined, the function $\Delta^{\frac{n-1}{2}}(f)^{\circ}_{s}$ is harmonic, $(f)^{\circ}_{s}$ being the so-called \textit{spherical value} (see Definition~\ref{spherical}).
Starting from these observations and from the explicit description of zonal harmonics in dimension 2 we
are able to prove that if $n=2m+1$ and $x=x_{0}+\underline{x}, y=y_{0}+\underline {y}$ belong to $\mathbb{R}^{n+1}$, then
\begin{equation*}
(\Delta_{y,n+1}\Delta_{x,n+1})^{m}\left[((xy^{c})^{k+2m})_{0}\right]=\frac{1}{2}\widetilde{\beta}_{k}^{m}Z^{m}_{k}(x,y).
\end{equation*}
In  Appendix~\ref{4D} we show explicitly the connection between our results and those in~\cite{perotticlifford}.

In the last section we obtain an equality similar to the previous one for any dimension involving,
instead of the double Laplacian, the action of a suitable fractional Laplacian together with the Kelvin inversion
(in a way analogous to Theorem~\ref{fueterqian} point \textit{(2)} or to formula~\eqref{2dim2}). The precise result is the following: for any $x\in\mR^{n+1}$ and any $y\in\mR^{n+1}\setminus\{0\}$, we have
\begin{equation*}
\mathcal{K}\left[\Delta^{\frac{n-1}{2}}((xy^{-1})^{-k})_{0}\right]=(n-1)!i^{n-1}\frac{k}{2k+n-1}Z_{k}^{\lambda}(x,y)
\end{equation*}

Eventually we compare these last two equations obtaining an analogue of Theorem~\ref{fueterqian} point \textit{(3)}:
let $n$ be an odd number, $n=2m+1$ and $x=x_{0}+\underline{x}, y=y_{0}+\underline {y}\in\mathbb{R}^{n+1}$, then
\begin{equation*}
(\Delta_{y,n+1}\Delta_{x,n+1})^{m}\left[((xy^{c})^{k+2m})_{0}\right]=\eta_{k}^{m}\mathcal{K}\left[\Delta^{m}_{n+1,x}((xy^{-1})^{-k})_{0}\right]
\end{equation*}
where $\eta_{k}^{m}$ is an explicit constant depending only on the dimension and on the degree (see Theorem~\ref{zonalscalar}).

\section{A ladder operator}\label{ladder}

In this section we give a realization of zonal harmonics in terms of a ladder operator which raises the degree of
a fixed polynomial. Starting from the constant polynomial equal to 1 (which is the zeroth zonal harmonic in any dimension), we are going to construct all the others.

We start from some basic identities.
Let $x=(x_0, x_{1},\dots,x_{n})\in\R^{n+1}$, then, for any $i=0,\dots, n$, we have,
\begin{equation*}\label{B1}
\partial_{x_{i}}|x|^{\alpha}=\partial_{x_{i}}(x_{1}^{2}+\dots+x_{n}^{2})^{\frac{\alpha}{2}}=\frac{\alpha}{2}|x|^{\alpha-2}2x_{i}=\alpha x_{i}|x|^{\alpha-2}.
\end{equation*}
From the previous equation we have
\begin{equation}\label{B2}
\partial_{x_{i}}(w)=\frac{y_{i}}{|x||y|}-\frac{\langle x,y\rangle}{|x|^{3}|y|}x_{i},
\end{equation}
where $w$ is defined as in formula~\eqref{angle}.
We now construct a ladder (raising) operator for Gegenbauer polynomials.
Consider the Kelvin inversion defined in formula~\eqref{kel} and the operator $\scalar{y,\nabla_x}=\sum_{j=1}^{n+1} y_j\diff{x_j}$ where $x,y\in\mR^{n+1}$.
We prove a raising property for the Gegenbauer polynomials. We start with the following initial step.

\begin{lemma}\label{lem1} In $\mR^{n+1}$ with $n>1$, it holds that
	\begin{align*}
		\left(\mathcal{K}\scalar{y,\nabla_x}\mathcal{K}\right)\left[1\right]=-C_1^\lambda(w)|x||y|=-\frac{\lambda}{\lambda+1}Z^{\lambda}_{1}(x,y),
	\end{align*}
	where $w=\frac{\scalar{x,y}}{\norm{x}\norm{y}}$ and $\lambda=\frac{n-1}{2}$ and $C_1^\lambda(t)$ is the Gegenbauer polynomial of degree $1$ and order $\lambda$.
	\end{lemma}
\begin{proof}
	We have 
		\begin{align*}
		\left(\mathcal{K}\scalar{y,\nabla_x}\mathcal{K}\right)\left[1\right]&=\mathcal{K}\scalar{y,\nabla_x}\frac{1}{\norm{x}^{n-1}}
			=\mathcal{K}\left[\sum_{j=1}^{n+1}y_j\diff{x_j}\frac{1}{\norm{x}^{n-1}}\right]\\
			&=\mathcal{K}\left[\sum_{j=1}^{n+1}-(n-1)x_jy_j\norm{x}^{-n-1}\right]\\
			&=-(n-1)\scalar{x,y}=C_1^\lambda\left(\frac{\scalar{x,y}}{\norm{x}\norm{y}}\right)|x||y|.
	\end{align*}
\end{proof}

Now we present a theorem that expresses any Gegenbauer polynomial (and hence any zonal harmonic), through the iterated action of the ladder operator $\mathcal{K}\scalar{y,\nabla_x}\mathcal{K}$ on the constant function equal to $1$.

\begin{theorem}
\label{th1}
The zonal harmonic of degree $k$ in $\mR^{n+1}$ can be obtained by 
\begin{align*}
	\left(\mathcal{K}\scalar{y,\nabla_x}\mathcal{K}\right)^k\left[1\right]=(-1)^{k}k!C_k^\lambda(w)(|x||y|)^{k}=(-1)^{k}k!\frac{\lambda}{\lambda+k}Z^{\lambda}_{k}(x,y),
\end{align*}
where $w=\frac{\scalar{x,y}}{\norm{x}\norm{y}}$, $\lambda=\frac{n-1}{2}$ and $C_k^\lambda(t)$ is the Gegenbauer polynomial of degree $k$ and order $\lambda$.
\end{theorem}


\begin{proof}
We will prove the statement by induction. By the definition of the Kelvin inversion~\eqref{kel} we have
\begin{align*}
	\left(\mathcal{K}\scalar{y,\nabla_x}\mathcal{K}\right)\left[ C_\ell^\lambda(w)(|x||y|)^{\ell}\right]=\mathcal{K}\left[\sum_{j=1}^{n+1}y_j\diff{x_j}\left(\frac{1}{\norm{x}^{n-1}}\norm{x}^{-\ell} C_\ell^\lambda(w)\norm{y}^\ell\right)\right].
\end{align*}
 Therefore, computing the partial derivatives we obtain
\begin{align*}
	\left(\mathcal{K}\scalar{y,\nabla_x}\mathcal{K}\right)\left[ C_\ell^\lambda(w)(|x||y|)^{\ell}\right]&=\mathcal{K}\left[\norm{y}^\ell\sum_{j=1}^{n+1}y_j\left((1-n-\ell)\norm{x}^{-n-1-\ell}x_jC_\ell^\lambda(w)\right.\right.\\
	&\left.\left.+\norm{x}^{1-n-\ell}\left(C_\ell^\lambda(w)\right)'\left(\diff{x_j}w\right)\right)\right].
\end{align*}
Using relation \eqref{G1} and~\eqref{B2} we have 
\begin{align*}
	\left(\mathcal{K}\scalar{y,\nabla_x}\mathcal{K}\right)\left[C_\ell^\lambda(w)(|x||y|)^{\ell}\right]&=\mathcal{K}\left[\norm{x}^{-n-\ell}\norm{y}^{\ell+1}\left((1-n-\ell)wC_\ell^\lambda(w)+2\lambda (1-w^2) C_{\ell-1}^{\lambda+1}(w)\right)\right].
\end{align*}
We apply relation \eqref{G3} on the first term and relation \eqref{g2} on the second one to get
\begin{align*}
&\left(\mathcal{K}\scalar{y,\nabla_x}\mathcal{K}\right)\left[C_\ell^\lambda(w)(|x||y|)^{\ell}\right]=\\
&=\mathcal{K}\left[\norm{x}^{-n-\ell}\norm{y}^{\ell+1}\left((1-n-\ell)\frac{1}{2(\ell+\lambda)}\left((\ell+1)C_{\ell+1}^\lambda(w)+(2\lambda+\ell-1)C_{\ell-1}^\lambda(w)\right)\right.\right.\\
&\quad\left.\left.+2\lambda\frac{1}{4\lambda(\ell+\lambda)}\left((2\lambda+\ell-1)(2\lambda+\ell)C_{\ell-1}^\lambda(w)-\ell(\ell+1)C_{\ell+1}^\lambda(w)\right)\right)\right]\\
&=\mathcal{K}\left[\norm{x}^{-n-\ell}\norm{y}^{\ell+1}\frac{1}{2(\ell+\lambda)}\left(\left((\ell+1)(1-n-\ell)-(\ell+1)\ell\right)C_{\ell+1}^\lambda(w)\right.\right.\\
&\quad\left.\left.+\left((1-n-\ell)(2\lambda+\ell-1)+(2\lambda+\ell-1)(2\lambda+\ell)\right)C_{\ell-1}^\lambda(w)\right)\right].
\end{align*}
Recalling that $\lambda=\frac{n-1}{2}$, the coefficient of $C_{\ell-1}^\lambda$ vanishes and we get
\begin{align*}
&\left(\mathcal{K}\scalar{y,\nabla_x}\mathcal{K}\right)\left[ C_\ell^\lambda(w)(|x||y|)^{\ell}\right]=\\
&=\mathcal{K}\left[\norm{x}^{-n-\ell}\norm{y}^{\ell+1}\frac{1}{2\ell+n-1}\left(\left((\ell+1)(1-n-\ell)-(\ell+1)\ell\right)C_{\ell+1}^\lambda\right)\right]\\
&=\mathcal{K}\left[-\norm{x}^{-n-\ell}\norm{y}^{\ell+1}(\ell+1)C_{\ell+1}^\lambda\right]\\
&=-(\ell+1)C_{\ell+1}^\lambda(|x||y|)^{\ell+1}.
\end{align*}
This holds for all $\ell=1,\cdots,k-1$ and the claim follows by induction and Lemma \ref{lem1} for the case $\ell=0$.
\end{proof}

\begin{remark}
In \cite{Eel} a version of Theorem \ref{th1} is proven for the reproducing kernel of the space of spherical monogenics, given by a sum of two Gegenbauer polynomials.
\end{remark}

\subsection{Poisson kernel}
We now apply Theorem~\ref{th1} to give a representation of the Poisson kernel in the unit ball of $\mR^{N}$.
First of all, notice that summing Theorem \ref{th1} over all degrees of homogeneity $k$, we get the exponential of the operator $-\mathcal{K}\scalar{y,\nabla_x}\mathcal{K}$. Then, we have the following corollary.
\begin{cor}\label{generating}
For any $x,y\in\mR^{n+1}$ such that $1-2\langle x,y\rangle +|x|^{2}|y|^{2}$ is different from zero and $|x||y|<1$, we have
\begin{align}
\label{cor}
	\left(\exp(-\mathcal{K}\scalar{y,\nabla_x}\mathcal{K})\right)\left[1\right]=\left(\mathcal{K}\exp(-\scalar{y,\nabla_x})\mathcal{K}\right)\left[1\right]=\sum_{k=0}^\infty C_k^\lambda(w)(|x||y|)^{k},
\end{align}
where $\lambda=\frac{n-1}{2}$ and $w=\frac{\scalar{x,y}}{\norm{x}\norm{y}}$.
\end{cor}
\begin{proof}
The right-hand side of~\eqref{cor} is the generating function of the Gegenbauer polynomials. This is known to equal
\begin{align}
\label{gen}
(1-2rw+r^2)^{-\lambda}=\sum_{k=0}^\infty C_k^\lambda(w)r^k,
\end{align}
where $r=\norm{x}\norm{y}$, see e.g.~\cite[Chapter IV.2]{SW}. This can be re-established by interpreting $\exp(-\scalar{y,\nabla_x})$ as a translation operator:   $\exp(-\scalar{y,\nabla_x})f =f(x-y)$.
We then have indeed 
\begin{align*}
	\left(\mathcal{K}\exp(-\scalar{y,\nabla_x})\mathcal{K}\right)\left[1\right]&=\left(\mathcal{K}\exp(-\scalar{y,\nabla_x})\right)\frac{1}{\norm{x}^{n-1}}\\
	&=\mathcal{K}\left(\frac{1}{\norm{x-y}^{n-1}}\right)=\frac{1}{\norm{x}^{n-1}}\frac{1}{\norm{\frac{x}{\norm{x}^2}-y}^{n-1}}\\
	&=\left(1-2\norm{x}\norm{y}\frac{\scalar{x,y}}{\norm{x}\norm{y}}+\norm{x}^2\norm{y}^2\right)^{-\lambda}=\left(1-2rw+r^2\right)^{-\lambda}.
\end{align*}
\end{proof}

Recall now, for instance from~\cite[formula  6.21, p. 122]{axler}, the formula for the (extended) Poisson kernel in $\mR^{N}$,
$$
P(x,y)=\frac{1-|x|^{2}|y|^{2}}{(1-2\langle x,y\rangle +|x|^{2}|y|^{2})^{N/2}},
$$
defined for all $x,y\in\mR^{N}$ such that the denominator is different from zero. 
Now, if $\lambda=\frac{N-2}{2}$, from~\cite[Theorem 5.33]{axler}, we have 
$$
P(x,y)=\sum_{k=0}^{\infty}Z^{\lambda}_{k}(x,y)=\sum_{k=0}^{\infty}\frac{k+\lambda}{\lambda}C^{\lambda}_{k}(w)(|x||y|)^{k}.
$$
One can find the Poisson kernel from the generating function~\eqref{gen} (with $\norm{x}=1$ and hence $r=\norm{y}<1$) by acting with $1+\frac{r}{\lambda}\partial_r$.

\begin{corollary}
Set $r=|x||y|$, then, the (extended) Poisson kernel of $\mR^{n+1}$ is given by
$$
P(x,y)=\left(1+\frac{r}{\lambda}\partial_r\right)\left(\mathcal{K}\exp(-\scalar{y,\nabla_x})\mathcal{K}\right)\left[1\right]
$$
\end{corollary}

\begin{proof}
If $r=|x||y|$, then clearly $\langle x,y\rangle=rw$. Thanks to Corollary~\ref{generating} we have that
$$
\left(1+\frac{r}{\lambda}\partial_r\right)\left(\mathcal{K}\exp(-\scalar{y,\nabla_x})\mathcal{K}\right)\left[1\right]=\left(1+\frac{r}{\lambda}\partial_r\right)(1-2rw+r^2)^{-\lambda}
$$
Then, we obtain the result from the following computations
\begin{align*}
	\left(1+\frac{r}{\lambda}\partial_r\right)\left(1-2rw+r^2\right)^{-\lambda}&=(1-2rw+r^2)^{-\lambda}-\lambda\frac{r}{\lambda}2(r-w)\left(1-2rw+r^2\right)^{-\lambda-1}\\
	&=(1-2rw+r^2)^{-\lambda-1}\left((1-2rw+r^2)-2r(r-w)\right)\\
	&=\frac{1-r^2}{(1-2rw+r^2)^{\lambda+1}}=P(x,y).
\end{align*}

\end{proof}

\section{Actions of iterated Laplacians on Gegenbauer polynomials}\label{iterated}


In this section we give general formulas for the action of Laplacians on Gegenbauer polynomials. This will allow us to connect the reproducing kernels of spherical harmonics for different dimensions with each other.

A first strategy to obtain such a result would be to explicitly compute the action of several Laplacians on Gegenbauer polynomials. It becomes however rapidly more complicated to keep track of all arising coefficients as is shown in Appendix~\ref{iterate}. Therefore we use a different path.

%


We start with a few lemmas.

\begin{lemma}
\label{Gegexp}
Assume $\lambda >0$. Then
\begin{equation}
\label{recgeg}
C^{\lambda}_{k+2m}(t) = \sum_{j=0}^m \alpha^{m, \lambda}_j C^{\lambda+m}_{k+2(m-j)}(t)
\end{equation}
with
\[
\alpha^{m, \lambda}_m = (-1)^m \frac{\Gamma(\lambda+m)}{\Gamma(\lambda)} \frac{\Gamma(\lambda+k+m+1)}{\Gamma(\lambda+k+2m +1)}.
\]
\end{lemma} 

\begin{proof}
According to~\eqref{G3bis} we have
\[
C^{\lambda}_{k+2m}(t) = \frac{\lambda}{\lambda+ k + 2m} \left(C^{\lambda+1}_{k+2m}(t)- C^{\lambda+1}_{k+2(m-1)}(t)   \right).
\]
Applying this formula recursively $m$ times to the right-hand side leads to~\eqref{recgeg}. The coefficient $\alpha^{m, \lambda}_m$ arises as the product of all used coefficients with minus sign in~\eqref{G3bis}. Indeed,
\begin{align*}
\alpha^{m, \lambda}_m& = (-1)^m \frac{\lambda}{\lambda + k +2m} \frac{\lambda+1}{\lambda + k +2m-1}  \ldots \frac{\lambda+m-1}{\lambda +k + m+1}\\
& = (-1)^m \frac{\Gamma(\lambda+m)}{\Gamma(\lambda)} \frac{\Gamma(\lambda+k+m+1)}{\Gamma(\lambda+k+2m +1)}.
\end{align*}
\end{proof}
We now state a similar result for Chebyshev polynomials.
\begin{corollary}
Let $T_{p}$ denote the Chebyshev polynomial of the first kind of degree $p$. Then, it holds
\begin{equation*}
2T_{k+2m}(t)=\sum_{j=0}^{m}\widehat{\alpha}_{j}^{m}C^{m}_{k+2(m-j)}(t),
\end{equation*}
where
$$
\widehat{\alpha}_{m}^{m}=(-1)^{m}\Gamma(m)\frac{\Gamma(k+m+1)}{\Gamma(k+2m)}.
$$
\end{corollary}
\begin{proof}
The proof can be performed as that of Lemma~\ref{Gegexp}, bearing in mind from formula~\eqref{chebygegen} that, $2T_{m}(t)=C^{1}_{m}(t)-C^{1}_{m-2}(t).$


\end{proof}

The following result comes from \cite{DBS}.
\begin{lemma}
\label{LapHarm}
Let $H_k \in \cH_k(\mR^N)$ be a spherical harmonic of degree $k$ in $\mR^N$. Then 
\[
\Delta^j \left(|x|^{2\ell} H_k(x) \right) = c^N_{j,\ell, k} |x|^{2\ell-2j} H_k(x)
\]
with
\[
c^N_{j,\ell, k} = \left\{ \begin{array}{ll}0& j > \ell \\
4^j \frac{\ell !}{(\ell-j)!} \frac{\Gamma(k + \ell + N/2)}{\Gamma(k+\ell-j+N/2)}& j \leq \ell.
 \end{array} \right.
\]
\end{lemma}

To lighten the notation, in the remainder of this section, the Laplacian in $\mR^{N}$ with respect to the variable $x$ will be
denoted by $\Delta_{x}=\Delta_{x,N}$ and analogously for $y$.

\begin{theorem}\label{thmiterated}
In $\mR^{2(\lambda+m)+2}$ we have
\[
(\Delta_{y}\Delta_{x})^{m}\left[(|x||y|)^{k+2m} C^{\lambda}_{k+2m}(w)\right]= \beta_{k}^{m, \lambda} \left( |x| |y| \right)^k C^{\lambda+m}_{k}(w) 
\]
where
\[
\beta_{k}^{m, \lambda} = (-1)^m 4^{2m} \Gamma(m-1)^2 \frac{\Gamma(\lambda+m)}{\Gamma(\lambda)} \frac{\Gamma(\lambda+k+2m+1)}{\Gamma(\lambda+k+m +1)}=\alpha_{m}^{m,\lambda}(c^{2(\lambda+m)+2}_{m,m,k})^{2}.
\]
and $w = \frac{\langle x,y\rangle}{|x||y|}$.
\end{theorem}

\begin{proof}
Recall that in $\mR^{2(\lambda+m)+2}$
\[
(|x| |y|)^j C^{\lambda+m}_j (w)\]
is a harmonic polynomial for all $j \in \mN$.

Now we have
\begin{align*}
(\Delta_{y}\Delta_{x})^{m}\left[(|x||y|)^{k+2m} C^{\lambda}_{k+2m}(w)\right] & =  \sum_{j=0}^m \alpha^{m, \lambda}_j (\Delta_{y}\Delta_{x})^{m} \left[(|x||y|)^{k+2m} C^{\lambda+m}_{k+2(m-j)}(w) \right]\\
& =  \sum_{j=0}^m \alpha^{m, \lambda}_j (\Delta_{y}\Delta_{x})^{m} \left[(|x||y|)^{2j} (|x||y|)^{k+2(m-j)} C^{\lambda+m}_{k+2(m-j)}(w) \right]\\
&=\alpha^{m, \lambda}_m (\Delta_{y}\Delta_{x})^{m} \left[(|x||y|)^{2m} (|x||y|)^{k} C^{\lambda+m}_{k}(w) \right]\\
&=  \alpha^{m, \lambda}_m \left( c^{2(\lambda + m)+2}_{m,m, k} \right)^2 (|x||y|)^{k} C^{\lambda+m}_{k}(w)
\end{align*}
where we used Lemma \ref{Gegexp} in the first line and Lemma \ref{LapHarm} in line 3 and 4.

The coefficient
\[
\beta_{k}^{m, \lambda} = \alpha^{m, \lambda}_m \left( c^{2(\lambda + m)+2}_{m,m, k} \right)^2
\]
follows immediately from the expressions for $\alpha^{m, \lambda}_m$ and $c^{2(\lambda + m)+2}_{m,m, k}$.
\end{proof}

We now use Theorem~\ref{thmiterated} to show how to construct zonal harmonic functions in any dimension,
starting from the simpler cases of $\mR^{2}$ and $\mR^{3}$. We start from dimension $3$ since the result is
a straightforward application of the formula just obtained for $\lambda=1/2$.

\begin{corollary}\label{odddim}
In $\mR^{2m+3}$, the zonal harmonic function of degree $k$ can be obtained as
\[
(\Delta_{y}\Delta_{x})^{m}\left[ Z^{1/2}_{k+2m}(x,y)\right]=\widetilde{\beta}_{k}^{m} Z^{1/2+m}_{k}(x,y), 
\]
where $Z^{1/2}_{p}$ denotes the zonal harmonic of degree $p$ in $3$ dimensions and 
\[
\widetilde{\beta}_{k}^{m}=(-1)^{m}\frac{\Gamma(m+1)\Gamma(2m+2)\Gamma(k+m+1)\Gamma(2k+4m+1)}{2(k+2m)(2k+2m+1)^{2}\Gamma(k+2m+1)\Gamma(2k+2m+1)}.
\]
\end{corollary}
\begin{proof}
The proof is a direct consequence of Theorem~\ref{thmiterated} together with formula~\eqref{zonalGegenbauer}.
\end{proof}

We now pass to all even dimensions.
\begin{corollary}\label{evendim}
In $\mR^{2m+2}$, the zonal harmonic function of degree $k$ can be obtained as
\[
(\Delta_{y}\Delta_{x})^{m}\left[ Z^{0}_{k+2m}(x,y)\right]=\widehat{\beta}_{k}^{m} Z^{m}_{k}(x,y), 
\]
where $Z^{0}_{p}$ denotes the zonal harmonic of degree $p$ in $2$ dimensions and
\begin{equation}\label{alpha}
\widehat{\beta}_{k}^{m}=(-1)^{m}\frac{4^{2 m} (k + 2 m) \Gamma(m+1)^3 \Gamma(k + 2 m+1)}{(k + m) \Gamma(k + m+1)}.
\end{equation}
\end{corollary}
\begin{proof}
The proof can be performed as that of Theorem~\ref{thmiterated} together with formulas~\eqref{zonalGegenbauer} and~\eqref{chebyshev}.
\end{proof}

\begin{remark}\label{simple}
If $y$ is fixed with $|y|=1$ then in Theorem~\ref{thmiterated}, Corollary~\ref{odddim} and Corollary~\ref{evendim} we only need
to consider the action of $\Delta_{x}$. In particular the three coefficients $\beta_{k}^{m,\lambda}$, $\widetilde{\beta}_{k}^{m}$ and $\widehat{\beta}_{k}^{m}$ simplify as follows.
In $\mR^{2(\lambda+m)+2}$ we have
\[
\Delta_{x}^{m}\left[|x|^{k+2m} C^{\lambda}_{k+2m}(w)\right]= (-1)^{m}4^{m}\frac{\Gamma(\lambda+m)}{\Gamma(\lambda)}\Gamma(m+1)  |x| ^k C^{\lambda+m}_{k}(w).
\]
In $\mR^{2m+3}$, the zonal harmonic function of degree $k$ can be obtained as
\[
\Delta_{x}^{m}\left[ Z^{1/2}_{k+2m}(x,y)\right]=(-1)^{m}\frac{2k+4m+1}{2k+2m+1}\Gamma(2m+2) Z^{1/2+m}_{k}(x,y).
\]
In $\mR^{2m+2}$, the zonal harmonic function of degree $k$ can be obtained as
\[
\Delta_{x}^{m}\left[ Z^{0}_{k+2m}(x,y)\right]=(-1)^{m}4^{m}\frac{k+2m}{k+m}(m!)^{2}Z^{m}_{k}(x,y).
\]

\end{remark}

\begin{remark}
Notice that similar results to the one proven in this section can be obtained using the operator $\frac{\Delta_{x}}{|y|^{2}}$ instead of $\Delta_{y}\Delta_{x}$.

%

\end{remark}

\section{Harmonicity of slice regular functions}\label{slicereg}

In the recent paper~\cite{perotticlifford}, A.~Perotti shows some harmonicity properties of slice regular functions, their connections with Clifford analysis  and with zonal harmonics in the particular case of real dimension 4.
Some of these properties were already exploited in~\cite{altavillabisi} to find a possible generalization of the Jensen formula in higher dimension (see also~\cite{perottiJensen}). Other results connecting the theory of slice regularity 
with the one of Fueter were given in~\cite{perottifueter}.
In this section we quickly review slice regularity in $\R_{n}$, the results of~\cite{perotticlifford} and add a
couple of results useful to connect this theory with the one developed in previous sections.
In particular, we will show how the so-called Fueter theorem applies to slice regular functions.

For any slice function it is possible to define its \textit{spherical value} and \textit{spherical derivative}.

\begin{definition}\label{spherical}
Let $f=\mathcal{I}(F):\Omega_{D}\to\R_{n}$ be any slice regular function. For any $x=\alpha+I\beta\in\Omega_{D}$, if $z=\alpha+i\beta\in D$, we define the \textit{spherical value} $f^{\circ}_{s}$
and the \textit{spherical derivative} $f'_{s}$ of $f$ as
\begin{align*}
f^{\circ}_{s}(x)&:=\frac{1}{2}(f(x)+f(x^{c}))=\mathcal{I}(F_{1})(x),\\
f'_{s}(x)&:=(x-x^{c})^{-1}(f(x)-f(x^{c}))=\mathcal{I}\left(\frac{F_{2}}{\Im(z)}\right).
\end{align*}
\end{definition}

The spherical derivative is initially only defined in non-real points. 
However, as it is shown in~\cite{ghiloniperotti}, the definition can be extended to $\R$. 
Both the spherical value and the spherical derivative of a slice regular functions are constant along
spheres of the form $\mathbb{S}_{\alpha+I\beta}:=\{\alpha+J\beta\,|\, J\in\mathbb{S}\}$.
\begin{remark}
Notice that if $f$ is a slice preserving function, then its spherical value is just its real part, i.e.: $f^{\circ}_{s}=(f)_{0}$.
\end{remark}

A fundamental result in the theory of slice regularity is the so-called \textit{Splitting Lemma}. 
As explained in~\cite[Section 2]{ghiloniperottipower}, the real algebra $\R_{n}$ has the \textit{splitting property}, i.e.: for any $J\in \mathbb{S}$, it is possible to find $h=2^{n-1}-1$ elements, $J_{1},\dots, J_{h}\in\mathcal{Q}$, such that $\{1,J,J_{1},JJ_{1},\dots,J_{h},JJ_{h}\}$ is a real vector basis of $\R_{n}$. Such 
a basis is called a \textit{splitting basis}. 

\begin{lemma}[Splitting Lemma]\label{splitting}
Let $J$ be any element of $\mathbb{S}$, $\{1,J,J_{1},JJ_{1},\dots,J_{h},JJ_{h}\}$ be a splitting basis for $\R_{n}$. Then, denoting $1=J_{0}$, the map
$$
\left(\mathcal{S}_{\R}(\Omega_{D})\right)^{2^{n}}\ni (f_{1,0},f_{2,0},\dots,f_{1,h},f_{2,h})\mapsto \sum_{\ell=0}^{h}(f_{1,\ell}J_{\ell}+f_{2,\ell}JJ_{\ell})\in\mathcal{S}(\Omega_{D}),
$$
is bijective. Therefore, for any given $f\in\mathcal{S}(\Omega_{D})$ there exist unique 
$f_{1,0},f_{2,0},\dots,f_{1,h},f_{2,h}\in\mathcal{S}_{\R}(\Omega_{D})$, such that
$$
f=\sum_{\ell=0}^{h}(f_{1,\ell}J_{\ell}+f_{2,\ell}JJ_{\ell}).
$$
\end{lemma}

This result is, in a sense, a rephrasing of the well-known \textit{splitting lemma} \cite{gentilistoppatostruppa}.
A first version in the quaternionic setting was given in \cite[Proposition 3.12]{C-GC-S}
The proof of the previous Lemma is the same as~\cite[proof of Lemma 6.11]{ghilonimorettiperotti}, suitably changing the indices and the notation.

Notice that a slice regular function $f$ is slice preserving, if and only if, in the previous lemma $f=f_{1,0}$.

A direct consequence of the Splitting Lemma is the Fueter theorem for slice regular functions restricted to the space of paravectors in any dimension. Moreover, information on harmonicity can be inferred from that. First of all notice that, given any slice regular function $f$, its restriction to the paravectors
satisfies
$$
f|_{\R^{n+1}}(x_{0}+\underline{x})=f^{\circ}_{s}(x_{0}+\underline{x})+\underline{x}f'_{s}(x_{0}+\underline{x}).
$$
Thanks to~\cite[Corollary 3.3 (a)]{perotticlifford}, for any $n\in\N$ given any slice regular function 
$f\in\mathcal{S}(\Omega)$, its restriction $f|_{\R^{n+1}}:\Omega\cap\R^{n+1}\to \R_{n}$ has the following property
\begin{equation}\label{cauchyspherical}
\bar D f=(1-n)f'_{s}.
\end{equation}

We are now in position to state the announced Fueter theorem for slice regular functions, in the same spirit of~\cite{qian}. The following result was partly known by experts but we did not find any explicit proof in the literature.

%

\begin{lemma}\label{fueterslice}
Let $n$ be a positive integer, $\Delta=\Delta_{x,n+1}$ be the $(n+1)$-dimensional Laplacian with respect to the variable $x$ and $f:\Omega\subset\R^{n+1}\to\R_{n}$ be (the restriction of) a 
slice regular function such that $\Delta^{\frac{n-1}{2}}f$ is a well-defined differentiable function. The following facts hold true.
\begin{enumerate}
\item The function $\Delta^{\frac{n-1}{2}}f$ is monogenic, i.e.: $\bar D \Delta^{\frac{n-1}{2}}f=0$.
\item The function $\Delta^{\frac{n-1}{2}}f$ is harmonic, i.e.: $\Delta \Delta^{\frac{n-1}{2}}f=0$.
\item The function $\Delta^{\frac{n-1}{2}}(f)^{\circ}_{s}$ is harmonic, i.e.: $\Delta \Delta^{\frac{n-1}{2}}(f)^{\circ}_{s}=0$.
\item If $n\geq 2$ then the function $\Delta^{\frac{n-3}{2}}(f)'_{s}$ is harmonic, i.e.: $\Delta \Delta^{\frac{n-3}{2}}(f)'_{s}=0$.
\end{enumerate}

\end{lemma}

\begin{proof}
For $n$ odd the result is true thanks to~\cite[Corollary 4.4]{perotticlifford}. We prove the result for  $n$ even.

Given any slice regular function $f$, it is possible~\cite{ghiloniperottistoppato} to expand 
its restriction on the space of paravectors in Laurent series of the form $f(x)=\sum_{k\in\mathbb{Z}}x^{k}a_{k}$, where, for $k>0$, we have
$$x^{-k}=\left(\frac{x^{c}}{xx^{c}}\right)^{k}=\frac{(x^{c})^{k}}{|x|^{2k}},$$
and hence $x^{-k}$ is a slice function regular away from zero. 
 Moreover, thanks to the Splitting Lemma~\ref{splitting}, the restriction of $f$ to the space of
 paravectors is a finite sum, with constant coefficient, of slice preserving regular functions $f_{h,\ell}$, for $h=1,2$ and $\ell=0,\dots, 2^{n-1}-1$, 
 i.e.: slice regular functions induced by complex intrinsic holomorphic functions. This splitting can be 
 performed by splitting each $a_{k}$ in its $2^{n}$ real components.
 To any of these functions $f_{h,\ell}$ the Fueter theorem holds true.
 Therefore, applying the Fueter Theorem in the formulation of~\cite[Section 4]{kou} to all the slice preserving regular functions
 $f_{h,\ell}$  that give the splitting of $f$, we have that $\Delta^{(n-1)/2}f$ is monogenic, proving \textit{(1)}. Moreover, since $\Delta=D\bar D$, then, clearly
$\Delta^{(n+1)/2}f=0$, proving \textit{(2)}.

Thanks to the commutation properties in formula~\eqref{commuting}, we have, using equation~\eqref{cauchyspherical}
\begin{equation}\label{(4)}
0=\bar D \Delta^{(n-1)/2}f=\Delta^{(n-1)/2} \bar D f=\Delta^{(n-1)/2} ((1-n)f'_{s}),
\end{equation}
proving \textit{(4)}.

We now pass to the spherical value. We know that
\begin{align*}
0&=\Delta \Delta^{(n-1)/2}f(x)=\Delta^{(n-1)/2}\Delta f(x)=\Delta^{(n-1)/2}\Delta (f^{\circ}_{s}(x)+\underline{x}f'_{s}(x))\\
&= \Delta^{(n-1)/2}\Delta(f^{\circ}_{s})+\Delta^{(n-1)/2}\Delta(\underline{x}f'_{s}(x)).
\end{align*}
Assume that $n=2m$. We now want to prove that $\Delta^{(n-1)/2}\Delta(\underline{x}f'_{s}(x))=\Delta^{1/2}\Delta^{m}(\underline{x}f'_{s}(x))=0$. It is possible to prove by induction that, for any $k\in\mN$
$$
\left[\Delta^{k},\underline{x}\right]=2k\partial_{\underline{x}}\Delta^{k-1}.
$$
Thus we have
$$
\Delta^{1/2}\Delta^{m}(\underline{x}f'_{s}(x))=\Delta^{1/2}\left[\underline{x}\Delta^{m}f'_{s}(x)+2m\partial_{\underline{x}}\Delta^{m-1}f'_{s}(x)\right].
$$
Now, thanks to formula~\eqref{(4)}, we have that $0=\Delta^{(n-1)/2} f'_{s}= \Delta^{m-\frac{1}{2}}f'_{s}$ and so $\Delta^{m}f'_{s}=0$.
Moreover, thanks to formula~\eqref{commuting}, we have proved that
$$
\Delta \Delta^{\frac{n-1}{2}}(f)^{\circ}_{s}=0.
$$
%
%
%
%
%
%
\end{proof}


We recall from~\cite[Theorem 4.6 (c)]{perotticlifford} that, for any $n\in\N$, the restriction of a slice
regular function to the space of paravectors $f:\Omega\cap\R^{n+1}\to \R_{n}$ satisfies
$$
\Delta_{x,n+1}f^{\circ}_{s}(x)=(1-n)\frac{\partial f'_{s}}{\partial x_{0}}(x).
$$
Now, since $(x-x^{c})$ is purely imaginary, then
$$
\frac{\partial}{\partial x_{0}}\left((x-x^{c})^{-1}(f(x)-f(x^{c}))\right)=(x-x^{c})^{-1}\frac{\partial}{\partial x_{0}}(f(x)-f(x^{c}))=\left(\frac{\partial f}{\partial x_{0}}\right)'_{s}=(\partial_{c}f)'_{s},
$$
where $\partial_{c}$ is the so-called \textit{slice derivative} (also denoted by $2\vartheta$ in~\cite{perotticlifford}) and
hence  \begin{equation}\label{derivative_value}
\Delta_{x.n+1}f^{\circ}_{s}(x)=(1-n)(\partial_{c}f)'_{s},
\end{equation}
confirming that, when $n=1$, the real part of a holomorphic function is harmonic.
Notice that, as usual, $\partial_{c}x^{k}=kx^{k-1}$. 

Given any positive integer $k$, the $k$-th power of the Clifford paravector is the restriction
of the slice regular function $x^{k}$. Moreover the spherical value $(x^{k})^{\circ}_{s}$ and the spherical derivative $(x^{k})'_{s}$ (restricted to $\R^{n+1}$), are homogeneous polynomials of the 
paravector variable $x_{0}+\underline{x}$ and are constant along \textit{parallels with
respect to the real line}, i.e.: spheres of the form $\mathbb{S}_{x}$. Starting from these observations,
in~\cite[Section 5]{perotticlifford}, Perotti shows, thanks to the harmonic properties of slice regular 
functions, that when $n=3$, and $k$ is any positive integer, the following facts hold true.
\begin{enumerate}
\item It holds: 
$$k(x^{k})'_{s}=Z_{k-1}^{1}(x,1).$$
\item For any $y\in\mR^{4}$, such that $|y|=1$ (and hence $y^{-1}=y^{c}$), it holds
$$Z_{k-1}^{1}(x,y)=Z_{k-1}^{1}(xy^{-1},1)=Z_{k-1}^{1}(xy^{c},1)=\frac{[xy^{c}-yx^{c}]^{-1}[(xy^{c})^{k}-(yx^{c})^{k}]}{k},$$ 
where the second equality is explicitly computed in Lemma~\ref{zydeco}.
\item It holds:
$$x^{k}=\frac{Z_{k}^{1}(x,1)}{k+1}+\frac{x-2x_{0}}{k}Z_{k-1}^{1}(x,1).$$ 
This can be seen as a particular case of Lemma~\ref{lemmaxyc}, recalling that $|x|^{2}=xx^{c}$ and that $-x^{c}=x-2x_{0}$.
\end{enumerate}

Since, for any $k$, the function $x^{k}$ is a slice preserving function, then the spherical value of $x^{k}$
is exactly its real part, i.e.:
$$
(x^{k})^{\circ}_{s}=(x^{k})_{0}.
$$
Moreover, since $x^{k}$ is slice preserving, thanks to formula~\eqref{derivative_value}, we have that
\begin{equation}\label{zonalvalue}
\Delta_{4}(x^{k+2})_{0}=\Delta_{4}(x^{k+2})^{\circ}_{s}=-2(\partial_{c}x^{k+2})'_{s}=-2(k+2)(x^{k})'_{s}=-2\frac{k+2}{k+1}Z_{k}^{1}(x,1).
\end{equation}

More in general, thanks to equation~\eqref{2dim1} and Corollary~\ref{evendim} we can now state the following general result
\begin{theorem}\label{zonalscalar}
Let $n+1$ be an even number and $n=2m+1$. If $x=x_{0}+\underline{x}$ and $y=y_{0}+\underline {y}$ belong to $\mathbb{R}^{n+1}$, then
\begin{equation}\label{zonalscalar1}
(\Delta_{y,n+1}\Delta_{x,n+1})^{m}\left[((xy^{c})^{k+2m})_{0}\right]=\frac{1}{2}\widetilde{\beta}_{k}^{m}Z^{m}_{k}(x,y),
\end{equation}
where $\widehat{\beta}_{k}^{m}$ is given in formula~\eqref{alpha}.
\end{theorem}

\begin{remark}
Notice that when $y=1$, formula~\eqref{zonalscalar1} simplifies as explained in Remark~\ref{simple} as
$$
\Delta_{x,n+1}^{m}\left[(x^{k+2m})_{0}\right]=\frac{1}{2}\left((-1)^{m}4^{m}\frac{k+2m}{k+m}(m!)^{2}\right)Z^{m}_{k}(x,1),
$$
which, for $m=1$ is exactly what we pointed out in formula~\eqref{zonalvalue}.
\end{remark}

\section{A further representation of zonal harmonics}\label{further}

In this last section we are going to give a further formula representing zonal harmonics. In particular we are going to
generalize formula~\eqref{2dim2} to any dimension by using fractional powers of the usual Laplacian and the main result of Section~\ref{ladder}.
In the whole present section we consider slice regular functions restricted to the set of paravectors.
%
%
As usual, let $\Delta$ denote the Laplacian in $\mR^{n+1}$ with respect to the variable $x$.
Since $x^{-k}=\left(\frac{x^{c}}{|x|^{2}}\right)^{k}$, if we write $x=x_{0}+I\beta\in\R^{n+1}$, then $((x^{c})^{k})_{0}$ is 
a finite sum of monomials of the form $x_{0}^{\ell}\beta^{h}$. Therefore, since the distributional Fourier transform is additive 
and it splits on functions of separable variables, we get that $\mathcal{F}(x^{-k})$ is a homogeneous polynomial
that is constant along spheres of the form $\mathbb{S}_{x}$.
Since $x^{-k}$ is homogeneous of degree $-k$, then the same is true for $(x^{-k})_{0}$.
From formula~\eqref{kelvinhomo} we know that, since $\Delta^{\frac{n-1}{2}}(x^{-k})_{0}$ is a homogenous
function of degree $1-n-k$, then $\mathcal{K}[\Delta^{\frac{n-1}{2}}(x^{-k})_{0}]$ is homogeneous of degree
$k$.
Therefore, thanks to the properties of zonal harmonics, for any integer $n\geq 0$ we have that $\mathcal{K}[\Delta^{\frac{n-1}{2}}(x^{-k})_{0}]$ equals, up to a multiplicative coefficient, the function $Z^{\frac{n-1}{2}}_{k}(x,1)$.
The aim of this section is to give an explicit expression of such coefficient in the case in which $1$ is replaced by
any paravector $y$ such that $|y|=1$. Then, we compare this last representation
with the one given in Section~\ref{slicereg}.

We start by adapting previous general considerations about the Fourier transform to our particular case.
Using formula~\eqref{gamma}, we have that

\begin{align*}
\frac{\gamma_{1,n}}{\gamma_{1,2}} &=\frac{i\Gamma\left(\frac{n+1}{2}\right)\big/\Gamma\left(\frac{1+n+1-n}{2}\right)}{i\Gamma\left(\frac{3}{2}\right)\big/\Gamma\left(\frac{1+n+1-2}{2}\right)}\\
&=\frac{\Gamma\left(\frac{n+1}{2}\right)\big/\Gamma\left(\frac{n}{2}\right)}{\Gamma\left(1\right)\big/\Gamma\left(\frac{3}{2}\right)}=2^{2-n}\Gamma(n)=2^{2-n}(n-1)!
\end{align*}

We start our computations with $y^{c}$ instead of $y^{-1}$. In this case we have, after some simple computation
$$
((xy^c)^{-k})_{0}=\frac{1}{|y|^{2(k-1)}}\frac{(-1)^{k-1}}{(k-1)!}\langle y,\nabla_x \rangle^{k-1}\left(\frac{\langle x, y\rangle}{|x|^{2}|y|^2}\right),
$$
where $\langle y, \nabla_x\rangle$ denotes the directional derivative with respect to $y$ and, for $n>1$, that
$$
\frac{\langle x,y\rangle}{|x|^{n+1}}=\frac{1}{1-n}\langle y,\nabla_x\rangle\mathcal{K}[1].
$$
Therefore, by using formula~\eqref{fourierpol}, and following the idea in~\cite{qian}, we have the next sequence of equalities:

\begin{align*}
\Delta^{\frac{n-1}{2}}((xy^c)^{-k})_{0} &= 
\frac{1}{|y|^{2(k-1)}}\frac{(-1)^{k-1}}{(k-1)!}\langle y,\nabla_x \rangle^{k-1}\mathcal{F}^{-1}\left((i|\xi|)^{n-1}\mathcal{F}\left(\left(\frac{\langle x, y\rangle}{|x|^{2}|y|^2}\right)\right)(\xi)\right)\\
&=\frac{1}{|y|^{2k}}\frac{(-1)^{k-1}}{(k-1)!}i^{n-1}\langle y,\nabla_x \rangle^{k-1}\mathcal{F}^{-1}\left(|\xi|^{n-1}\gamma_{1,n}\frac{\langle \xi, y\rangle}{|\xi|^{n+1}}
\right)\\
&=\frac{1}{|y|^{2k}}\frac{(-1)^{k-1}}{(k-1)!}i^{n-1}\frac{\gamma_{1,n}}{\gamma_{1,2}}\langle y,\nabla_x \rangle^{k-1}\left(\frac{\langle x, y\rangle}{|x|^{n+1}}\right)\\
&=\frac{1}{|y|^{2k}}\frac{i^{n-1}}{n-1}\frac{(-1)^{k}}{(k-1)!}\frac{\gamma_{1,n}}{\gamma_{1,2}}\langle y,\nabla_x \rangle^{k}\mathcal{K}[1]\\
&=\frac{1}{|y|^{2k}}\frac{i^{n-1}}{n-1}\frac{(-1)^{k}}{(k-1)!}2^{2-n}(n-1)!\langle y,\nabla_x \rangle^{k}\mathcal{K}[1]\\
&=\frac{1}{|y|^{2k}}(n-2)!i^{n-1}\frac{(-1)^{k}}{(k-1)!}\langle y,\nabla_x \rangle^{k}\mathcal{K}[1].
\end{align*}

Therefore
\begin{equation}\label{y^{c}}
\mathcal{K}\left[\Delta^{\frac{n-1}{2}}((xy^c)^{-k})_{0}\right]=\frac{1}{|y|^{2k}}(n-2)!i^{n-1}\frac{(-1)^{k}}{(k-1)!}\left(\mathcal{K}\langle y,\nabla_x \rangle\mathcal{K}\right)^{k}[1]
\end{equation}

If now $y$ is any element different from zero, we have
$$
\left((xy^{-1})^{-k}\right)_{0}=\left(\left(x\frac{y^{c}}{|y|^{2}}\right)^{-k}\right)_{0}=\left((xy^{c})^{-k}|y^{2k}|\right)_{0}=|y|^{2k}\left((xy^{c})^{-k}\right)_{0}.
$$
Thanks to the last equality, formula~\eqref{y^{c}} becomes
\begin{equation*}
\mathcal{K}\left[\Delta^{\frac{n-1}{2}}((xy^{-1})^{-k})_{0}\right]=(n-2)!i^{n-1}\frac{(-1)^{k}}{(k-1)!}\left(\mathcal{K}\langle y,\nabla_x \rangle\mathcal{K}\right)^{k}[1].
\end{equation*}
Moreover, we recall from Theorem~\ref{th1} that,
$$
\frac{\left(\mathcal{K}\langle y,\nabla_x\rangle \mathcal{K}\right)^k}{k!}\left[1\right]=(-1)^{k}C_k^\lambda\left(\frac{\langle x,y\rangle}{|x||y|}\right)(|x||y|)^k
$$
Hence, recalling formula~\eqref{zonalGegenbauer}, we have proved the following result.
\begin{theorem} For any $x\in\mR^{n+1}$ and any $y\in\mR^{n+1}\setminus\{0\}$, we have
\begin{equation}\label{formulakelvinproduct}
\mathcal{K}\left[\Delta^{\frac{n-1}{2}}((xy^{-1})^{-k})_{0}\right]=(n-1)!i^{n-1}\frac{k}{2k+n-1}Z_{k}^{\lambda}(x,y).
\end{equation}
\end{theorem}


%
\begin{remark}
Notice that for $n=1$ (and so when $\mR_{n}=\mC$), formula~\eqref{formulakelvinproduct} becomes exactly formula~\eqref{2dim2}.
\end{remark}

The last result of this paper is a reinterpretation of Theorem~\ref{fueterqian} point (3).

\begin{theorem}\label{zonalscalar}
Let $n+1$ be an even number and $n=2m+1$. Let the variables $x=x_{0}+\underline{x}$ and $y=y_{0}+\underline {y}$ belong to $\mathbb{R}^{n+1}$. The one has
\begin{equation*}
(\Delta_{y,n+1}\Delta_{x,n+1})^{m}\left[((xy^{c})^{k+2m})_{0}\right]=\eta_{k}^{m}\mathcal{K}\left[\Delta^{m}_{n+1,x}((xy^{-1})^{-k})_{0}\right]
\end{equation*}
where $\eta_{k}^{m}$ is given by,
\begin{equation*}
\eta_{k}^{m}=\frac{4^{2 m} (k + 2 m) \Gamma(m+1)^3 \Gamma(k + 2 m+1)}{k(2m)! \Gamma(k + m+1)}.
\end{equation*}
\end{theorem}
\begin{proof}
We recall from equation~\eqref{zonalscalar1} that
\begin{equation*}
(\Delta_{y,n+1}\Delta_{x,n+1})^{m}\left[((xy^{c})^{k+2m})_{0}\right]=\frac{1}{2}\widehat{\beta}_{k}^{m}Z^{m}_{k}(x,y),
\end{equation*}
where,
\begin{equation*}
\frac{1}{2}\widehat{\beta}_{k}^{m}=\frac{1}{2}(-1)^{m}\frac{4^{2 m} (k + 2 m) \Gamma(m+1)^3 \Gamma(k + 2 m+1)}{(k + m) \Gamma(k + m+1)}.
\end{equation*}
Moreover, from equation~\eqref{formulakelvinproduct} we have that,
\begin{equation*}
Z_{k}^{m}(x,y)= (-1)^{-m}\frac{2(k+m)}{k(2m)!}\mathcal{K}\left[\Delta^{m}((xy^{-1})^{-k})_{0}\right]   
\end{equation*}
Hence
\begin{equation*}
\eta_{k}^{m}=\frac{4^{2 m} (k + 2 m) \Gamma(m+1)^3 \Gamma(k + 2 m+1)}{k(2m)! \Gamma(k + m+1)}
\end{equation*}

\end{proof}

\begin{remark}
When $m=0$ (i.e. the dimension is 2), we have that for any $k\in\mathbb{N}$, $\eta_{k}^{0}=1$ and hence
$$
\left((xy^{c})^{k}\right)_{0}=\mathcal{K}\left[\left((xy^{-1})^{-k}\right)_{0}\right],
$$
exactly as established in the introduction.
\end{remark}

\begin{remark} 
If $y$ is fixed and such that $|y|=1$ (so that $y^{-1}=y^{c}$), the coefficient $\eta_{k}^{m}$ simplifies a bit as explained in Remark~\ref{simple}.
In this case we have
\begin{equation*}
\Delta_{x,n+1}^{m}\left[((xy^{c})^{k+2m})_{0}\right]=4^{m}\frac{k+2m}{k(2m)!}(m!)^{2}\mathcal{K}\left[\Delta^{m}_{n+1,x}((xy^{c})^{-k})_{0}\right]
\end{equation*}
%
%
Hence, if $m=1$ (i.e. $n+1=4$), denoting by $\Delta$ the Laplacian in $\mR^{4}$ with respect to the variable $x$, we get
\begin{equation*}
\Delta \left[(x^{k+2})_{0}\right]=\eta_{k}^{1}\mathcal{K}\left[\Delta(x^{-k})_{0}\right]=2\frac{k+2}{k}\mathcal{K}\left[\Delta(x^{-k})_{0}\right]
\end{equation*}
This last formula can be considered an analog of~\cite[Proposition 5.1 (c)]{perotticlifford} for the spherical derivative.
\end{remark}

%
%

\appendix

\section{Explicit computations of the action of a double Laplacian}\label{iterate}
In this appendix we show that computing explicitly the double Laplacian considered in Section~\ref{iterated} might be 
much more difficult than our shortcuts.
As in previous sections, we will use the following notation $w=\frac{\langle x, y\rangle}{|x||y|}$.
For the convenience of the reader in what follows, all the computation in this section are done in
dimension $N$. 

We will make use of formula~\eqref{B2} together with the following
\begin{equation}\label{B2bis}
\partial^{2}_{x_{i}}(w)=\frac{1}{|x|^{3}|y|}\left(-3x_{i}y_{i}-\langle x,y\rangle+3\langle x,y\rangle\frac{ x_{i}^{2}}{|x|^{2}}\right),
\end{equation}
and
\begin{equation}\label{B3}
\left(\partial_{x_{i}}(w)\right)^{2}=\left(\frac{y_{i}}{|x||y|}-\frac{\langle x,y\rangle}{|x|^{3}|y|}x_{i}\right)^{2}=\frac{1}{(|x||y|)^{2}}\left(y_{i}^{2}-2\langle x,y\rangle \frac{x_{i}y_{i}}{|x|^{2}}+\langle x,y \rangle^{2}\frac{x_{i}^{2}}{|x|^{4}}\right).
\end{equation}
We start from the first partial derivative:
$$
\partial_{x_{i}}\left(C_{k}^{\lambda}(w)|x|^{\ell}\right)=2\lambda C^{\lambda+1}_{k-1}(w)|x|^{\ell}\left(\frac{y_{i}}{|x||y|}-\frac{\langle x,y\rangle}{|x|^{3}|y|}x_{i}\right)+\ell C^{\lambda}_{k}(w)|x|^{\ell-2}x_{i}.
$$

Using formulas~\eqref{B2bis} and~\eqref{B3}, the second derivative follows:
\begin{align*}
\partial^{2}_{x_{i}}\left(C_{k}^{\lambda}(w)|x|^{\ell}\right)=&2\lambda(2\lambda+2)C^{\lambda+2}_{k-2}(w)|x|^{\ell}\left(\frac{y_{i}}{|x||y|}-\frac{\langle x,y\rangle}{|x|^{3}|y|}x_{i}\right)^{2}+\\
&+(2\lambda)\ell C^{\lambda+1}_{k-1}(w)|x|^{\ell-2}x_{i}\left(\frac{y_{i}}{|x||y|}-\frac{\langle x,y\rangle}{|x|^{3}|y|}x_{i}\right)+\\
&+2\lambda C^{\lambda+1}_{k-1}(w)|x|^{\ell}\left(-\frac{x_{i}y_{i}}{|x|^{3}|y|}-2\frac{x_{i}y_{i}}{|x|^{3}|y|}-\frac{\langle x,y\rangle}{|x|^{3}|y|}+3\frac{\langle x,y\rangle x_{i}^{2}}{|x|^{5}|y|}\right)+\\
&+2\lambda\ell C^{\lambda+1}_{k-1}(w)|x|^{\ell-2}x_{i}\left(\frac{y_{i}}{|x||y|}-\frac{\langle x,y\rangle}{|x|^{3}|y|}x_{i}\right)+\\
&+C^{\lambda}_{k}(w)\ell(\ell-2)|x|^{\ell-4}x_{i}^{2}+\ell C^{\lambda}_{k}(w)|x|^{\ell-2}.
\end{align*}

We are able to compute the Laplacian with respect to the variable $x$:
\begin{align*}
\Delta_{N,x}\left[C^{\lambda}_{k}(w)|x|^{\ell}\right]=&\sum_{i=1}^{N}\partial_{x_{i}}^{2}(C^{\lambda}_{k}(w)|x|^{\ell})= N\ell C^{\lambda}_{k}(w)|x|^{\ell-2}+\ell(\ell-2)C^{\lambda}_{k}(w)|x|^{\ell-2}+\\
&+2\lambda C^{\lambda+1}_{k-1}(w)|x|^{\ell}\left(\frac{\langle x,y\rangle}{|x||y|}\right)(1-N)+\\
&+2\lambda(2\lambda+2)C^{\lambda+2}_{k-2}(w)|x|^{\ell}\left(\frac{1}{|x|^{2}}-\frac{\langle x,y\rangle^{2}}{|x|^{4}|y|^{2}}\right)\\
=& \ell(N+\ell-2)C^{\lambda}_{k}(w)|x|^{\ell-2}+2\lambda(1-N)wC^{\lambda+1}_{k-1}(w)|x|^{\ell-2}+\\
&+2\lambda(2\lambda+2)(1-w^{2})C^{\lambda+2}_{k-2}(w)|x|^{\ell-2}\\
=&2\lambda(2\lambda+2-n)wC^{\lambda+1}_{k-1}(w)|x|^{\ell-2}+[\ell(N+\ell-2)-k(k+2\lambda)]C^{\lambda}_{k}(w)|x|^{\ell-2}.
\end{align*}
Now, thanks to formula~\eqref{G3}, we get
\begin{align*}
\Delta_{N,x}\left[C^{\lambda}_{k}(w)|x|^{\ell}\right]=& |x|^{\ell-2}\{(2\lambda+2-N)(2\lambda C^{\lambda+1}_{k-2}(w)+kC^{\lambda}_{k}(w))+\\
&+(\ell(N+\ell-2)-k(k+2\lambda)C^{\lambda}_{k}(w))\}\\
=& |x|^{\ell-2}\{(2\lambda)(2\lambda+2-N)C^{\lambda+1}_{k-2}(w)+(\ell-k)(N+k+\ell-2)C^{\lambda}_{k}(w)\}.
\end{align*}

The previous formula, in the special case $\ell=k$, gives:
\begin{align*}
\Delta_{N,x}\left[C^{\lambda}_{k}(w)|x|^{k}\right]=& 2\lambda(2\lambda+2-N)wC^{\lambda+1}_{k-1}(w)|x|^{k-2}+k(N-2-2\lambda)C^{\lambda}_{k}(w)|x|^{k-2}\\
=&(2\lambda+2-N)|x|^{k-2}[2\lambda wC^{\lambda+1}_{k-1}(w)-kC^{\lambda}_{k}(w)]\\
=&(2\lambda+2-N)|x|^{k-2}(2\lambda)C^{\lambda+1}_{k-2}(w),
\end{align*}
where in the last equality we have used formula~\eqref{G4}.
Therefore $C^{\lambda}_{k}(w)|x|^{k}$ is harmonic if and only if $2\lambda+2-N=0$ if and only if $\lambda=\frac{N-2}{2}$.

We now compute the Laplacian with respect to the variable $y$:
\begin{align*}
&\Delta_{N,y}\Delta_{N,x}\left[C^{\lambda}_{k}(w)|x|^{k}|y|^{k}\right] = \Delta_{N,y}\left[2\lambda(2\lambda+2-N)C^{\lambda+1}_{k-2}(w)|x|^{k-2}|y|^{k}\right]\\
&=  2\lambda(2\lambda+2-N)|x|^{k-2}\Delta_{N,y}\left[C^{\lambda+1}_{k-2}(w)|y|^{k}\right]\\
&=2\lambda(2\lambda+2-N)(|x||y|)^{k-2}\left[2(N+2k-4)C^{\lambda+1}_{k-2}(w)+(2\lambda+2)(2\lambda +4-N)C^{\lambda+2}_{k-4}(w)\right]\\
&=(|x||y|)^{k-2}\left\{2\lambda(2\lambda+2-N)2(N+2k-4)C^{\lambda+1}_{k-2}(w)\right.\\
&\quad\left.+2\lambda(2\lambda+2)(2\lambda+2-N)(2\lambda+4-N)C^{\lambda+2}_{k-4}(w)\right\}.
\end{align*}

For general integer powers of the norm $|x| |y|$ we have, after similar computations
\begin{prop}
\label{dirlap}
One has, for $\ell\in\mN$
\begin{align*}
&\Delta_{N,y}\Delta_{N,x}\left[C^{\lambda}_{k}(w)|x|^{\ell}|y|^{\ell}\right]=\\
&=|x|^{\ell-2}\left\{ 2\lambda(2\lambda+2-N)\Delta_{N,y}\left[C^{\lambda+1}_{k-2}(w)|y|^{\ell}\right]\right.\\
&\quad\left.+(\ell-k)(N+k+\ell-2)|x|^{\ell-2}\Delta_{N,y}\left[C^{\lambda}_{k}(w)|y|^{\ell}\right] \right\}\\
&=(|x||y|)^{\ell-2}\left\{ (\ell-k)(N+k+\ell-2)\left[2\lambda(2\lambda+2-N)C^{\lambda+1}_{k-2}(w)+(\ell-k)(N+k+\ell-2)C^{\lambda}_{k}(w)\right]\right.\\
&\quad\left.+2\lambda(2\lambda+2-N)\left[(2\lambda+2)(2\lambda+4-N)C^{\lambda+2}_{k-4}(w)+(\ell-k+2)(N+k+\ell-4)C^{\lambda+1}_{k-2}(w)\right] \right\}
\end{align*}
\end{prop}

Notice that when $N=6$, $\lambda=1$ and $k=\ell$, Proposition \ref{dirlap} yields
\begin{align*}
&\Delta_{6,y}\Delta_{6,x}\left[C^{1}_{k}(w)|x|^{k}|y|^{k}\right]=\\
&=(|x||y|)^{k-2}\left\{2(2+2-6)2(6+2k-4)C^{1+1}_{k-2}(w)+2(2+2)(2+2-6)(2+4-6)C^{1+2}_{k-4}(w)\right\}\\
&=-16(1+k)C^{2}_{k-2}(w)(|x||y|)^{k-2}
\end{align*}
which, up to the prefactor, is exactly the zonal harmonic of degree $k-2$. Moreover this last result coincides with 
the one given in Theorem~\ref{thmiterated}.

\section{Explicit computations in Dimension 4}\label{4D}
In this second appendix we prove a generalization of some equalities presented in Section~\ref{slicereg} and
introduced by A.~Perotti~\cite{perotticlifford}.
In the whole appendix $x$ and $y$ will denote two paravectors, i.e. $x=x_{0}+\underline{x}, y=y_{0}+\underline{y}\in\R^{n+1}$.

If we extend the usual notation for the spherical value and spherical derivative as
$$
((xy^{c})^{k})_{s}^{\circ}=((xy^{c})^{k})_{0}=\frac{1}{2}((xy^{c})^{k}+(yx^{c})^{k}),\qquad ((xy^{c})^{k})_{s}'=[xy^{c}-yx^{c}]^{-1}[(xy^{c})^{k}-(yx^{c})^{k}],
$$
then, by direct computation, we have 
\begin{equation}\label{zydeco}
(xy^{c})^{k}=((xy^{c})^{k})_{s}^{\circ}+\underline{(xy^{c})}((xy^{c})^{k})_{s}',
\end{equation}
where $\underline{(xy^{c})}\in\mR^{n}$ is the vector part of $xy^{c}$.
In the first lemma we give a computational proof of~\cite[Theorem 5.1 (b)]{perotticlifford}.
\begin{lemma}
For any $k\geq 0$, we have the following equality
$$
((xy^{c})^{k+1})_{s}'=\frac{1}{k+1}Z^{1}_{k}(x,y).
$$
\end{lemma}

\begin{proof}
First of all, notice that $(xy^{c})^{c}=yx^{c}$, hence
$$(xy^{c})_{0}=(yx^{c})_{0}=\langle x,y\rangle,\qquad xy^{c}-yx^{c}=\underline{(xy^{c})}-\underline{(yx^{c})}=2\underline{(xy^{c})}=-2\underline{(yx^{c})}.$$
Moreover $\underline{(xy^{c})}^{2}=-|\underline{(xy^{c})}|^{2}=-|\underline{(xy^{c})}^{c}|^{2}=\underline{(yx^{c})}^{2}$. Therefore
\begin{align*}
(xy^{c})^{k+1}-(yx^{c})^{k+1} &= ((xy^{c})_{0}+\underline{(xy^{c})})^{k+1}+((yx^{c})_{0}+\underline{(yx^{c})})^{k+1}\\
	&=\sum_{r=0}^{k+1}\binom{k+1}{r}\left[(xy^{c})_{0}^{k+1-r}(xy^{c})_{v}^{r}-(yx^{c})_{0}^{k+1-r}\underline{(yx^{c})}^{r}\right]\\
	&= \sum_{r=0}^{k+1}\binom{k+1}{r}\left[(xy^{c})_{0}^{k+1-r}[\underline{(xy^{c})}^{r}-\underline{(yx^{c})}^{r}]\right].
\end{align*}
Now, for $r=2h$, we have 
$$
\underline{(xy^{c})}^{2h}-\underline{(yx^{c})}^{2h}=(-|\underline{(xy^{c})}|^{2})^{h}-(-|\underline{(xy^{c})}^{c}|^{2})^{h}=(-|\underline{(xy^{c})}|^{2})^{h}-(-|\underline{(yx^{c})}^{c}|^{2})^{h}=0,
$$
while, for $r=2h+1$, we have
\begin{align*}
\underline{(xy^{c})}^{2h+1}-\underline{(yx^{c})}^{2h+1}&=(-|\underline{(xy^{c})}|^{2})^{h}\underline{(xy^{c})}-(-|\underline{(xy^{c})}|^{2})^{h}\underline{(yx^{c})}\\
&=(-|\underline{(xy^{c})}|^{2})^{h}[\underline{(xy^{c})}-\underline{(yx^{c})}]=(-1)^{h}|\underline{(xy^{c})}|^{2h}2\underline{(xy^{c})}.
\end{align*}

Therefore, we obtain
\begin{equation*}
(xy^{c})^{k+1}-(yx^{c})^{k+1}= \sum_{h=1}^{[k/2]}\binom{k+1}{2h+1}\langle x,y \rangle^{k-2h}(-1)^{h}|\underline{(xy^{c})}|^{2h}2\underline{(xy^{c})}.
\end{equation*}

Now we use the following equalities:
$$
|xy^{c}|^{2}=(xy^{c})(xy^{c})^{c}=|x|^{2}|y|^{2}=|(xy^{c})_{0}+\underline{(xy^{c})}|^{2}=\langle x,y\rangle^{2}+|(\underline{xy^{c})}|^{2},
$$
and so
$$
|\underline{(xy^{c})}|^{2h}=(|x|^{2}|y|^{2}-\langle x,y\rangle^{2})^{h}=\sum_{\ell=0}^{h}\binom{h}{\ell}(-1)^{\ell}\langle x,y\rangle^{2\ell}(|x|^{2}|y|^{2})^{h-\ell}.
$$
Summarizing, (since $(-1)^{\ell}=(-1)^{-\ell}$, for any integer $\ell$),we get,
\begin{align*}
&[xy^{c}-yx^{c}]^{-1}[(xy^{c})^{k+1}-(yx^{c})^{k+1}]=\\
&=(2\underline{(xy^{c})})^{-1}\left[\sum_{h=1}^{[k/2]}\binom{k+1}{2h+1}\langle x,y \rangle^{k-2h}(-1)^{h}\right.\left.\left(\sum_{\ell=0}^{h}\binom{h}{\ell}(-1)^{\ell}\langle x,y\rangle^{2\ell}(|x|^{2}|y|^{2})^{h-\ell}\right)2\underline{(xy^{c})}\right]\\
&= \sum_{h=1}^{[k/2]}\binom{k+1}{2h+1}\langle x,y \rangle^{k-2h}(-1)^{h}\left(\sum_{\ell=0}^{h}\binom{h}{\ell}(-1)^{\ell}\langle x,y\rangle^{2\ell}(|x|^{2}|y|^{2})^{h-\ell}\right)\\
&=\sum_{h=1}^{[k/2]}\sum_{\ell=0}^{h}\binom{k+1}{2h+1}\binom{h}{\ell}(-1)^{h-l}\langle x,y\rangle^{k-2(h-\ell)}(|x|^{2}|y|^{2})^{h-\ell}\\
&=\sum_{\ell=0}^{[k/2]}\sum_{h=\ell}^{[k/2]}\binom{k+1}{2h+1}\binom{h}{\ell}(-1)^{h-l}\langle x,y\rangle^{k-2(h-\ell)}(|x|^{2}|y|^{2})^{h-\ell}.
\end{align*}
Finally, using the change of variables $h-l=r$ we obtain that 
\begin{align*}
&[xy^{c}-yx^{c}]^{-1}[(xy^{c})^{k+1}-(yx^{c})^{k+1}]=\\
&=\sum_{\ell=0}^{[k/2]}\sum_{r=0}^{[k/2]-\ell}\binom{k+1}{2(r+\ell)+1}\binom{r+\ell}{\ell}(-1)^{r}\langle x,y\rangle^{k-2r}(|x|^{2}|y|^{2})^{r}\\
&=\sum_{r=0}^{[k/2]}\left(\sum_{\ell=0}^{[k/2]-r}\binom{k+1}{2(r+\ell)+1}\binom{r+\ell}{\ell}\right)(-1)^{r}\langle x,y\rangle^{k-2r}(|x|^{2}|y|^{2})^{r}.
\end{align*}
Now the general coefficient of our sum is easily identified with
\begin{equation*}
\sum_{\ell=0}^{[k/2]-r}\binom{k+1}{2(r+\ell)+1}\binom{r+\ell}{\ell} =
\frac{\Gamma(k+2)}{\Gamma(2r+2)\Gamma(k-2r+1)}\sum_{\ell=0}^{[k/2]-r}\frac{(r-\frac{k}{2})_\ell(r-\frac{k-1}{2})_\ell}{(r+\frac32)_\ell}\frac{1}{\ell!}
\end{equation*}
where $(.)_{\ell}$ denotes the Pochhammer symbol. The last sum is the hypergeometric function $_{2}F_{1}(r-\frac{k}{2}, r-\frac{k-1}{2}; r+\frac32;1)$ evaluated in 1. Note that this hypergeometric function terminates as either $r-\frac{k}{2}$ or $r-\frac{k-1}{2}$ is a negative integer. Its value is known due to Gauss's hypergeometric theorem as
$$
_{2}F_{1}\left(r-\frac{k}{2}, r-\frac{k-1}{2}; r+\frac32;1\right)=\frac{\Gamma(r+\frac32)\Gamma(k-r+1)}{\Gamma(\frac{k}{2}+\frac32)\Gamma(\frac{k}{2}+1)}.
$$
We therefore have
\begin{align*}
\sum_{\ell=0}^{[k/2]-r}\binom{k+1}{2(r+\ell)+1}\binom{r+\ell}{\ell}&=
\frac{\Gamma(k+2)}{\Gamma(\frac{k}{2}+\frac32)\Gamma(\frac{k}{2}+1)}\cdot\frac{\Gamma(r+\frac32)}{\Gamma(2r+2)}\cdot\frac{\Gamma(k-r+1)}{\Gamma(k-2r+1)}\\
&=\frac{2^{k+1}}{\sqrt{\pi}}\frac{\sqrt{\pi}}{2^{2r+1}\Gamma(r+1)}\frac{\Gamma(k-r+1)}{\Gamma(k-2r+1)}=2^{k-2r}{k-r\choose r},
\end{align*}
where we used the duplication formula $\Gamma(2z)=\frac{2^{2z-1}}{\sqrt{\pi}}\Gamma(z)\Gamma(z+\frac12)$ on the first two factors.\\
The result is exactly the general coefficient of the Gegenbauer polynomial $C^{1}_{k}$ (see formula~\eqref{explicitgegenbauer}), and we know that in dimension 4 the parameter $\lambda$ equals one (see formula~\eqref{zonalGegenbauer}), i.e.
$$
Z_{k}^{1}(x,y)=(k+1)C_{k}^{1}\left(\frac{\langle x,y\rangle}{|x||y|}\right)(|x||y|)^{k}.
$$
\end{proof}

We now want to generalize~\cite[Corollary 5.2]{perotticlifford} but before that, we need the following lemma.

\begin{lemma}
For any $k\geq0$, it holds
\begin{equation}\label{eq1}
((xy^{c})^{k})_{0}=((xy^{c})^{k+1})_{s}'-\langle x,y\rangle((xy^{c})^{k})_{s}'.
\end{equation}
\end{lemma}

\begin{proof} We have,
\begin{align*}
&((xy^{c})^{k+1})_{s}'-\langle x,y\rangle((xy^{c})^{k})_{s}'=\\
&=[xy^{c}-yx^{c}]^{-1}[(xy^{c})^{k+1}-(yx^{c})^{k+1}]-
\langle x,y\rangle[xy^{c}-yx^{c}]^{-1}[(xy^{c})^{k}-(yx^{c})^{k}]\\
&=[xy^{c}-yx^{c}]^{-1}[(xy^{c})^{k+1}-(yx^{c})^{k+1}-
\langle x,y\rangle((xy^{c})^{k}-(yx^{c})^{k})]\\
&=[xy^{c}-yx^{c}]^{-1}[(xy^{c}-\langle x,y\rangle)(xy^{c})^{k}-(yx^{c}-\langle x,y\rangle)(yx^{c})^{k}]\\
&=[xy^{c}-yx^{c}]^{-1}[(xy^{c}-\langle x,y\rangle)(xy^{c})^{k}]-[xy^{c}-yx^{c}]^{-1}(yx^{c}-\langle x,y\rangle)(yx^{c})^{k}]\\
&=[2\underline{(xy^{c})}]^{-1}[\underline{(xy^{c})}(xy^{c})^{k}]-[-2\underline{(yx^{c})}]^{-1}\underline{(yx^{c})}(yx^{c})^{k}]\\
&=\frac{1}{2}(xy^{c})^{k}+\frac{1}{2}(yx^{c})^{k}.
\end{align*}
\end{proof}

Now, using formula~\eqref{eq1}, we get
$$
((xy^{c})^{k})_0=\frac{1}{k+1}Z^{1}_{k}(x,y)-\langle x,y\rangle\frac{1}{k}Z^{1}_{k-1}(x,y),
$$
or, equivalently, if $w=\frac{\langle x,y\rangle}{|x||y|}$,
\begin{equation*}
((xy^{c})^{k})_0=C^{1}_{k}(w)(|x||y|)^{k}-\langle x,y\rangle C^{1}_{k-1}(w)(|x||y|)^{k-1}=\left[C^{1}_{k}(w)-w C^{1}_{k-1}(w)\right](|x||y|)^{k}.
\end{equation*}
\begin{lemma}\label{lemmaxyc}
For any $k\geq0$, it holds
$$
(xy^{c})^{k+1}=\frac{xy^{c}}{k+1}Z^{1}_{k}(x,y)-\frac{|x|^{2}|y|^{2}}{k}Z^{1}_{k-1}(x,y),
$$
or, equivalently
$$
(xy^{c})^{k+1}=\left[xy^{c}C^{1}_{k}(w)-|x||y|C^{1}_{k-1}(w)\right](|x||y|)^{k}.
$$\end{lemma}
\begin{proof}
Thanks to formula~\eqref{zydeco},
\begin{align*}
(xy^{c})^{k}&=\frac{1}{k+1}Z^{1}_{k}(x,y)-\langle x,y\rangle\frac{1}{k}Z^{1}_{k-1}(x,y)+\underline{(xy^{c})}\frac{1}{k}Z^{1}_{k-1}(x,y)\\
&=\frac{1}{k+1}Z^{1}_{k}(x,y)+(\underline{(xy^{c})}-\langle x,y\rangle)\frac{1}{k}Z^{1}_{k-1}(x,y)\\
&=\frac{1}{k+1}Z^{1}_{k}(x,y)-yx^{c}\frac{1}{k}Z^{1}_{k-1}(x,y),
\end{align*}
therefore, recalling that $(xy^{c})^{-1}=(yx^{c})/|x|^{2}|y|^{2}$, then
$$
(xy^{c})^{k+1}=\frac{xy^{c}}{k+1}Z^{1}_{k}(x,y)-\frac{|x|^{2}|y|^{2}}{k}Z^{1}_{k-1}(x,y).
$$
\end{proof}

\bibliographystyle{amsplain}

\end{document}